\documentclass[12pt]{amsart}

\usepackage{amssymb, amsmath, amsthm}
\usepackage[margin=1in]{geometry}
\usepackage[dvipdfmx]{graphicx,xcolor} 
\usepackage{tikz}

\allowdisplaybreaks

\numberwithin{equation}{section}

\theoremstyle{plain}
\newtheorem{thm}{Theorem}[section]
\newtheorem{prop}[thm]{Proposition}
\newtheorem{lem}[thm]{Lemma}
\newtheorem{cor}[thm]{Corollary}
\newtheorem*{referthmA}{Theorem A}
\newtheorem*{referthmB}{Theorem B}

\newtheorem*{rem*}{Remark}
\theoremstyle{definition}
\newtheorem{defn}[thm]{Definition}
\newtheorem{rem}[thm]{Remark}

\newcommand{\N}{\mathbb{N}}
\newcommand{\R}{\mathbb{R}}

\newcommand{\Z}{\mathbb{Z}}

\newcommand{\calA}{\mathcal{A}}

\newcommand{\calF}{\mathcal{F}}

\newcommand{\calS}{\mathcal{S}}

\newcommand{\supp}{\mathrm{supp}\, }
\newcommand{\op}{\mathrm{Op}}

\newcommand{\ichi}{\mathbf{1}}

\newcommand{\veck}{\boldsymbol{k}}
\newcommand{\vecm}{\boldsymbol{m}}
\newcommand{\vecs}{\boldsymbol{s}}

\newcommand{\vecz}{\boldsymbol{z}}
\newcommand{\vecxi}{\boldsymbol{\xi}}

\newcommand{\vecnu}{\boldsymbol{\nu}}
\newcommand{\vecmu}{\boldsymbol{\mu}}
\newcommand{\vectau}{\boldsymbol{\tau}}
\newcommand{\vecell}{\boldsymbol{\ell}}

\newcommand{\II}{I\hspace{-3pt}I}

\begin{document}
\title[$S_{0,0}$ class symbols of limited smoothness]
{Multilinear pseudo-differential operators with 
$S_{0,0}$ class symbols of limited smoothness} 

\author[T. Kato]{Tomoya Kato}

\address[T. Kato]
{Division of Pure and Applied Science, 
Faculty of Science and Technology, Gunma University, 
Kiryu, Gunma 376-8515, Japan}

\email[T. Kato]{t.katou@gunma-u.ac.jp}

\date{\today}

\thanks{This work was supported by JSPS KAKENHI, 
Grant Numbers 20K14339.}

\begin{abstract}
We consider the boundedness of 
the multilinear pseudo-differential operators 
with symbols in the multilinear H\"{o}rmander class $S_{0,0}$.
The aim of this paper is to
discuss smoothness conditions for symbols to 
assure the boundedness between local Hardy spaces.
\end{abstract}

\maketitle

\section{Introduction}\label{secIntro}

First of all, 
the letter $N$ which is mentioned in this article 
is understood to be a positive integer
unless the contrary is explicitly stated.

For a bounded measurable function 
$\sigma = \sigma (x, \xi_1, \dots, \xi_N)$ on $(\R^n)^{N+1}$, 
the ($N$-fold) multilinear pseudo-differential operator $T_{\sigma}$ 
is defined by  
\[
T_{\sigma}(f_1,\dots,f_N)(x)
=\frac{1}{(2\pi)^{Nn}}
\int_{(\R^n)^N}e^{i x \cdot(\xi_1+\dots+\xi_N)}
\sigma (x, \xi_1, \dots, \xi_N) 
\prod_{j=1}^{N} \widehat{f_j}(\xi_j)
\, d\xi_1 \dots d\xi_N
\]
for $x\in \R^n$ and 
$f_1,\dots,f_N \in \calS(\R^n)$.
The function $\sigma$ is called the symbol of the operator 
$T_{\sigma}$.

The subject of the present paper is 
to investigate the boundedness of 
the multilinear pseudo-differential operators on
several function spaces.
In stating this, 
we use the following terminology with a slight abuse. 
Let $X_1,\dots, X_N$, and $Y$ be function spaces on $\R^n$ 
equipped with quasi-norms 
$\|\cdot \|_{X_j}$ and $\|\cdot \|_{Y}$, 
respectively.  
If there exists a constant $C$
such that
\begin{equation}\label{boundedness-XjY}
\|T_{\sigma}(f_1,\dots, f_N)\|_{Y}
\le C 
\prod_{j=1}^{N} \|f_j\|_{X_j}, \quad
f_j\in \calS \cap X_j, \quad
j=1,\dots,N,
\end{equation}
then we say that 
$T_{\sigma}$ is bounded from 
$X_1 \times \cdots \times X_N$ to $Y$. 
The smallest constant $C$ of 
\eqref{boundedness-XjY} 
is denoted by 
$\|T_{\sigma}\|_{X_1 \times \cdots \times X_N \to Y}$. 
If $\calA$ is a class of symbols,  
we denote by $\mathrm{Op}(\calA)$
the class of all operators $T_{\sigma}$ 
corresponding to $\sigma \in \calA$. 
If $T_{\sigma}$ is bounded from 
$X_1 \times \cdots \times X_N$ to $Y$ 
for all $\sigma \in \calA$, 
then we write 
$\mathrm{Op}(\calA) 
\subset B (X_1 \times \cdots \times X_N \to Y)$. 
For the spaces $X_{j}$ and $Y$,
we consider the Lebesgue space $L^p$,
the Hardy space $H^p$,
the local Hardy space $h^p$, 
and the spaces
$BMO$ and $bmo$.
The definitions of these spaces will be collected
in Subsection \ref{subsecFunctionSpaces}.

Notice that, if $T_{\sigma}$ is bounded from
$X_1 \times \cdots \times X_N$ to $Y$ 
in the sense given above, 
then, in many cases,  
we can extend the definition of 
$T_{\sigma}$ defined for $f_j \in \calS (\R^n)$
to that for general $f_j \in X_j$
and can prove that 
\eqref{boundedness-XjY} holds 
for all $f_j \in X_j$ using some limiting argument.

In this article, we focus on 
the H\"ormander symbol class of $S_{0,0}$-type.
We recall that
the class $S_{0,0}^{m} (\R^n,N)$, $m \in \R$,
consists of all smooth functions 
$\sigma$ on $(\R^n)^{N+1}$ such that
\[
| \partial^{\alpha_0}_x \partial^{\alpha_1}_{\xi_1} \cdots \partial^{\alpha_N}_{\xi_N} 
\sigma(x,\xi_1,\dots,\xi_N) |
\le C_{\alpha_0,\alpha_1,\dots,\alpha_N}
( 1 + |\xi_1| +\cdots + |\xi_N| )^{m}
\]
holds for all multi-indices 
$\alpha_0, \alpha_1, \dots, \alpha_N \in (\N_0)^n
= (\{ 0,1,2,\dots\})^{n}$.
The linear case, $N=1$,
is the widely known H\"ormander class 
and the following is a classical boundedness result:

\begin{referthmA}
Let $0 < p \le \infty$ and $m\in \R$. 
Then, the boundedness
\begin{equation*}
\op ( S^{m}_{0,0}(\R^n, 1) ) \subset
B(h^{p}  \to h^{p})
\end{equation*}
holds
if and only if 
\begin{equation*}
m \le 
\min \Big\{ \frac{n}{p}, \frac{n}{2} \Big\} 
- \max \Big\{ \frac{n}{p}, \frac{n}{2} \Big\},    
\end{equation*}
where, if $p=\infty$,
$h^{p}$ should be replaced by $bmo$.
\end{referthmA}

The ``if" part of this result for $p=2$ was proved
by Calder\'on and Vaillancourt in \cite{CV},
and then it was generalized to the case $1< p < \infty$
by Fefferman in \cite{Fefferman} and Coifman and Meyer in \cite{CM-Ast}.
Finally, the boundedness for the full range $0 < p \le \infty$ was obtained
by Miyachi in \cite{Miyachi-MN} and 
P\"aiv\"arinta and Somersalo in \cite{PS}.
For the ``only if'' part,
see, for instance, \cite[Section 5]{Miyachi-MN} and \cite[Theorem 1.5]{KMT-JFA}.

The study of the multilinear case, $N \ge 2$,
originated with the paper \cite{BT-2004} by B\'enyi and Torres,
where they showed that, for $N=2$ and 
for $1 \le p, p_1, p_2 < \infty$ with $1/p = 1/p_1 + 1/p_2$,
$x$-independent symbols in $S^{0}_{0,0}(\R^n, 2)$ 
do not always give rise to bounded operators
from $L^{p_1} \times L^{p_2}$ to $L^p$. 
Then, the condition of $m \in \R$
for which the multilinear pseudo-differential operators 
with symbols in the class
$S^{m}_{0,0}(\R^n, N)$
can be bounded among local Hardy spaces 
was investigated.
More precisely, the following holds:

\begin{referthmB}
Let $N\ge 2$, 
$0 < p, \, p_1,\, \dots, \, p_N \le \infty$, 
$1/p= 1/p_1 + \dots + 1/p_N$, and $m\in \R$. 
Then, the boundedness 
\begin{equation*}
\op( S^{m}_{0,0}(\R^n, N) ) \subset
B(h^{p_1} \times \cdots \times h^{p_N} \to h^{p})
\end{equation*}
holds if and only if 
\begin{equation*}
m \le 
\min \Big\{ \frac{n}{p}, \frac{n}{2} \Big\} 
- \sum_{j =1}^{N} 
\max \Big\{ \frac{n}{p_j}, \frac{n}{2} \Big\},
\end{equation*}
where, if $p_j = \infty$ for some $j\in \{1, \dots, N\}$,
the corresponding $h^{p_j}$ can be replaced by $bmo$. 
\end{referthmB}

The case $N=2$ was proved by Miyachi and Tomita \cite{MT-IUMJ}
and the case $N \ge 3$ by Miyachi, Tomita, and the author \cite{KMT-JFA}.
For the preceding results considering the subcritical case,
see the papers by Michalowski, Rule, and Staubach \cite{MRS} and
by B\'enyi, Bernicot, Maldonado, Naibo, and Torres \cite{BBMNT}.
Quite recently, a generalization of Theorem B for $N=2$
considering boundedness  on Sobolev spaces
was shown by Shida \cite{Shida}.

Remark that,
in Theorems A and B,
much smoothness
is implicitly assumed for symbols.
In the rest of this section, 
we shall consider
smoothness conditions for symbols
to assure the boundedness,
which is our interest of the present paper.
We first recall the linear case.
In Miyachi \cite{Miyachi-MN},
it was shown that
the smoothness condition of symbols
assumed in Theorem A
can be relaxed to,
roughly speaking, 
the smoothness up to 
$\min\{n/p, n/2\}$ for the space variable $x$
and 
$\max\{n/p, n/2\}$ for the frequency variable $\xi_1$.
Moreover, it might be worth mentioning that
these values are partially sharp
(see \cite[Section 5]{Miyachi-MN}).
Some results on this direction 
can be also found in, 
for instance,
Boulkhemair \cite{Boulkhemair},
Coifman and Meyer \cite{CM-Ast}, 
Cordes \cite{Cordes}, 
Hwang \cite{Hwang},
Muramatu \cite{Muramatu}, 
and
Sugimoto \cite{Sugimoto-JMSJ}
for $p=2$
and Tomita \cite{Tomita-AM}
for $0 < p < \infty$.
For the multilinear case,
in \cite{KMT-JPDOA, KMT-JMSJ},
it was shown that,
for the case $2/N \le p \le 2$ and $2 \le p_1, \dots, p_N \le \infty$,
the assumptions of the smoothness up to 
$n/2$ for each space and frequency variables
are sufficient to have the boundedness in Theorem B.
See also
Herbert and Naibo \cite{HN-2014, HN-2016}
for the preceding results.

The purpose of this paper is to extend
the partial result on the multilinear case stated above
to the whole range of the exponents 
$0 < p, p_1, \dots, p_N \le \infty$.
We shall determine
the smoothness conditions of symbols
for the boundedness in Theorem B
as weak as possible.
Before stating our main theorem,
we introduce a Besov type class 
to measure the smoothness of symbols.
In order to define this class,
we use a partition of unity as follows.
We take $\psi_{0}, \psi \in \calS (\R^n)$ 
satisfying that 
$\supp \psi_{0} \subset 
\{ \xi \in \R^n : |\xi| \leq 2 \}$,
$\supp \psi \subset 
\{ \xi \in \R^n : 1/2 \leq |\xi| \leq 2 \}$,
and
$\psi_{0}+\sum_{k\in\N} \psi ( 2^{-k} \cdot ) = 1$,
and denote $\psi_k := \psi ( 2^{-k} \cdot )$ for $k \in \N$.
We call this $\{\psi_k\}_{k\in \N_0}$ a  
Littlewood--Paley partition of unity on $\R^n$. 
Moreover, we write as 
$\vecxi = (\xi_1, \dots, \xi_N) \in (\R^n)^{N}$
and
$\langle \xi \rangle = (1 + |\xi|^2)^{1/2}$,
$\xi \in \R^d$,
to shorten the notations.

\begin{defn} \label{def-main-thm-1}
Let $N \ge 2$, $m \in \R$, and $t \in (0, \infty]$
and let $\{\psi_{k}\}_{ k\in \N_0 }$ be
a Littlewood--Paley partition of unity on $\R^n$.
For 
$\veck = (k_0, k_1, \dots, k_N) \in (\N_0)^{N+1}$,
$\vecs = (s_0, s_1, \dots, s_N) \in [0, \infty)^{N+1}$, and 
$\sigma = \sigma (x, \vecxi) \in L^{\infty} ((\R^n)^{N+1})$,
we write
$\vecs\cdot \veck
= \sum_{j=0}^{N} s_{j} k_{j}$
and
\begin{equation*}
\Delta_{\veck} 
\sigma (x, \vecxi)
=
\psi_{k_0} ( D_x )
\psi_{k_1} ( D_{\xi_1} )
\dots
\psi_{k_N} ( D_{\xi_N} )
\sigma (x, \vecxi) .
\end{equation*}
We denote by 
$S^{m}_{0,0} (\vecs, t ; \R^n, N)$
the set of all $\sigma \in L^{\infty} ((\R^n)^{N+1})$
such that the quasi-norm
\begin{align*} 
\|\sigma\|_{ S^{m}_{0,0} (\vecs,t; \R^n, N) } 
=
\bigg\{
\sum_{ \veck \in (\N_0)^{N+1} }
\big( 2^{ \vecs\cdot \veck } \big)^{t} \,
\big\| \|
\langle \vecxi \rangle^{-m}
\Delta_{\veck} 
\sigma (x, \vecxi) \|_{ L^2_{ul,\vecxi} ((\R^{n})^{N}) }
\big\|_{ L^{\infty}_{x} (\R^{n}) } ^{t}
\bigg\}^{1/t}
\end{align*}
is finite,
with a usual modification when $t=\infty$.
\end{defn}

Here, the space $L^{2}_{ul}$ is the uniformly local $L^{2}$ space,
which includes $L^\infty$ (see Subsection \ref{subsecFunctionSpaces}).
Using the class in Definition \ref{def-main-thm-1},
the main theorem of the present paper
reads as follows.

\begin{thm}\label{main-thm-1}
Let $N\ge 2$, 
$0<p, \, p_1,\, \dots, \, p_N \le \infty$, and
$1/p = 1/p_1 + \dots + 1/p_N$.
If
\begin{equation*}
m =
\min \Big\{ \frac{n}{p}, \frac{n}{2} \Big\} 
- \sum_{j =1}^{N} 
\max \Big\{ \frac{n}{p_j}, \frac{n}{2} \Big\} 
\end{equation*}
and
\[
s_0 = \min \Big\{ \frac{n}{p}, \frac{n}{2} \Big\} ,
\quad
s_j = \max \Big\{ \frac{n}{p_j}, \frac{n}{2} \Big\},
\quad j=1,\dots,N,
\]
then
\[
\op\left( S^{m}_{0,0} \big( \vecs, \min\{1,p\}; \R^n, N \big) \right) \subset
B(h^{p_1} \times \dots \times h^{p_N} \to h^{p}) ,
\] 
where, if $p_j = \infty$ for some $j\in \{1, \dots, N\}$,
the corresponding $h^{p_j}$ can be replaced by $bmo$. 
\end{thm}

We end this section with noting the organization of this paper.
In Section \ref{secPre}, 
we collect some notations which will be used throughout this paper
and give the definitions and properties of some function spaces.
In Section \ref{secMainResult}, 
we first display the key statements, 
Theorem \ref{main-thm-2} and Proposition \ref{main-prop},
which contains the essential part of Theorem \ref{main-thm-1}.
Then, we prove Theorem \ref{main-thm-2} 
and also consider the boundedness for 
symbols with classical derivatives.
After preparing several lemmas for the proof of 
Proposition \ref{main-prop} 
in Section \ref{seclemmasforProof},
we actually give its proof in Section \ref{secProof}.
In Section \ref{secSharp},
we consider the sharpness 
of the order $m$ and the smoothness $s_0,s_1\dots,s_N$
stated in Theorem \ref{main-thm-1}.

\section{Preliminaries}\label{secPre}


\subsection{Notations}
We denote by $Q$ the $n$-dimensional 
unit cube $[-1/2,1/2)^n$ and 
we write $\ell Q = [-\ell/2, \ell/2)^n$, $\ell > 0$.
Then, the cubes $\ell \tau + \ell Q$, $\tau \in \Z^n$, are mutually 
disjoint and constitute a partition of the 
Euclidean space $\R^n$.
This implies that
integral of a function on $\R^{n}$ 
is written as 
\begin{equation}\label{cubicdiscretization}
\int_{\R^n} f(x)\, dx = 
\sum_{\nu \in \Z^{n}} \int_{\ell Q} f(x+ \ell\nu)\, dx 
\end{equation}
for $\ell > 0$.
We denote by $B_R$ the closed ball in $\R^n$
of radius $R>0$ centered at the origin.
We denote by $\ichi_\Omega$ 
the characteristic function of a set $\Omega$.
For $1 \leq p \leq \infty$, $p^\prime$ is 
the conjugate number of $p$ defined by 
$1/p + 1/p^\prime =1$.

For two nonnegative functions $A(x)$ and $B(x)$ defined 
on a set $X$, 
we write $A(x) \lesssim B(x)$ for $x\in X$ to mean that 
there exists a positive constant $C$ such that 
$A(x) \le CB(x)$ for all $x\in X$. 
We often omit to mention the set $X$ when it is 
obviously recognized.  
Also $A(x) \approx B(x)$ means that
$A(x) \lesssim B(x)$ and $B(x) \lesssim A(x)$.

The symbols $\calS (\R^d)$ and 
$\calS^\prime(\R^d)$ denote 
the Schwartz class of rapidly 
decreasing smooth functions
and 
the space of tempered distributions 
on $\R^d$, respectively. 
The Fourier transform and the inverse 
Fourier transform of $f \in \calS(\R^d)$ are defined by
\begin{align*}
\mathcal{F} f  (\xi) 
= \widehat {f} (\xi) 
= \int_{\R^d}  e^{-i \xi \cdot x } f(x) \, dx,
\quad
\mathcal{F}^{-1} f (x) 
= \frac{1}{(2\pi)^d} \int_{\R^d}  e^{i x \cdot \xi } f( \xi ) \, d\xi,
\end{align*}
respectively. For a Schwartz function $f (x,\xi_{1},\dots,\xi_{N})$,  
$x,\xi_{1},\dots,\xi_{N} \in\R^{n}$, we denote the partial Fourier transform 
with respect to the $x$ and $\xi_{j}$ variables by
$\calF_0$ and $\calF_{j}$, $j=1,\dots,N$, respectively. 
We also write the Fourier transform on 
$(\R^{n})^{N}$ for the $\xi_{1},\dots, \xi_{N}$ variables as
$\calF_{1,\dots,N} = \calF_1 \dots \calF_{N}$.
For $m \in \calS^\prime (\R^n)$, 
we defined 
$m(D) f = \calF^{-1} [m \widehat{f} \,]$ 
and
use the notation 
$m(D)f(x)=m(D_x)f(x)$ 
when we indicate which variable is considered.
For $g \in \mathcal{S} (\mathbb{R}^n) \setminus \{0\}$,
we denote the short-time Fourier transform of 
$f \in \mathcal{S}^\prime (\mathbb{R}^n) $ with respect to $g$ by 
\begin{equation} \label{STFT}
V_g f (x,\xi) = 
\int_{\mathbb{R}^n} e^{ - i \xi \cdot t } \ \overline {g (t-x)} \, f(t) \, d t.
\end{equation}

\subsection{Function spaces} \label{subsecFunctionSpaces}
For a measurable subset $E \subset \R^d$, 
the Lebesgue space $L^p (E)$, $0<p\le \infty$, 
is the set of all those 
measurable functions $f$ on $E$ such that 
$\| f \|_{L^p(E)} = 
( \int_{E} \big| f(x) \big|^p \, dx )^{1/p} < \infty $
if $0< p < \infty$ 
or   
$\| f \|_{L^\infty (E)} 
= \operatorname{ess \, sup}_{x\in E} |f(x)| < \infty$ 
if $p = \infty$. 
If $E=\R^{n}$, 
we usually write $L^p $ for 
$L^p (\R^{n})$. 
The uniformly local $L^2$ space, denoted by 
$L^2_{ul} (\R^d)$, consists of 
all those measurable functions $f$ on 
$\R^d$ such that 
\begin{equation*}
\| f \|_{L^2_{ul} (\R^d) } 
= \sup_{\nu \in \Z^d}
\| f(x+\nu) 
\|_{L^{2}_{x} ([-1/2, 1/2)^d)}
< \infty .
\end{equation*}

For a countable set $K$,
the sequence space $\ell^q (K)$, $0< q \le \infty$, 
is the set of all those 
complex sequences $a=\{a_k\}_{k\in K}$ 
such that 
$ \| a \|_{ \ell^q (K)} 
= ( \sum_{ k \in K } | a_k |^q)^{ 1/q } 
<\infty$ 
if $0< q < \infty$  
or 
$\| a \|_{ \ell^\infty (K) } 
= \sup_{k \in K} |a_k| < \infty$ 
if $q = \infty$.
If $K=\Z^{n}$, 
we usually write $ \ell^q $ for 
$\ell^q (\Z^{n})$. 

Let $X,Y,Z$ be function spaces.  
We use the notation 
$\| f \|_{X} = \| f(x) \|_{X_{x}} $ 
when we indicate which variable is measured.
We denote the mixed norm by
\begin{equation*}
\label{normXYZ}
\| f (x,y,z) \|_{ X_x Y_y Z_z } 
= 
\Big\| \big\| \| f (x,y,z) 
\|_{ X_x } \big\|_{ Y_y } \Big\|_{ Z_z }.
\end{equation*} 
(Pay special attention to the order 
of taking norms.)
For $X, Y, Z$, we consider $L^{p}$ or $\ell^{p}$. 

Let $\phi \in \calS(\R^n)$ be such that
$\int_{\R^n}\phi(x)\, dx \neq 0$ and 
let 
$\phi_t(x)=t^{-n}\phi(x/t)$ for $t>0$. 
The space $H^p= H^p(\R^n)$, $0<p\leq\infty$,  
consists of
all $f \in \calS'(\R^n)$ such that 
$\|f\|_{H^p}=\|\sup_{0<t<\infty}|\phi_t*f|\|_{L^p}
<\infty$. 
The space $h^p= h^p(\R^n)$, $0<p\leq\infty$,  
consists of
all $f \in \calS'(\R^n)$ such that 
$\|f\|_{h^p}=\|\sup_{0<t<1}|\phi_t*f|\|_{L^p}
<\infty$.
It is known that $H^p$ and $h^p$ 
do not depend on the choice of the function $\phi$ 
up to the equivalence of quasi-norm. 
Obviously $H^p \subset h^p$. 
If $1 < p \leq \infty$, 
then 
$H^p = h^p = L^p$ with equivalent norms. 
In more details,
see, for instance,
Goldberg \cite{Goldberg}.

The space $BMO=BMO(\R^n)$ consists of
all locally integrable functions $f$ on $\R^n$ 
such that
\[
\|f\|_{BMO}
=\sup_{R}\frac{1}{|R|}
\int_R |f(x)- f_{R}|\, dx
<\infty,
\]
where $f_R=|R|^{-1}\int_R f(x) \, dx$ 
and $R$ ranges over all the cubes in $\R^n$.
The space $bmo=bmo(\R^n)$ consists of
all locally integrable functions $f$ on $\R^n$ 
such that
\[
\|f\|_{bmo}
=\sup_{|R| \le 1}\frac{1}{|R|}
\int_{R}|f(x)-f_R|\, dx
+\sup_{|R|\geq1}\frac{1}{|R|}
\int_R |f(x)|\, dx
<\infty,
\]
where $R$ denotes cubes in $\R^n$.
Obviously, 
$L^{\infty} \subset bmo \subset BMO$ holds.
Also,
$BMO$ satisfies that
$\|\lambda f\|_{BMO} = |\lambda| \|f\|_{BMO}$,
$\lambda \in \R$,
and
$\|f(\lambda \cdot)\|_{BMO} = \|f\|_{BMO}$,
$\lambda>0$
(see, for instance, Grafakos \cite[Section 3.1.1]{Grafakos-M}).

Let $\kappa\in \calS(\R^n)$ 
be a function such that the support of $\kappa$ is compact and
\begin{equation*} 
\sum _{k\in \mathbb{Z}^{n}} \kappa(\xi-k) = 1, 
\quad \xi \in \R^n.  
\end{equation*}
For $0< p, q \leq \infty$ and $s\in \R$, 
the Wiener amalgam space $W^{p,q}_{s}$ is defined to be the 
set of all $f \in \calS'(\R^n)$ such that 
the quasi-norm 
\begin{equation*}
\|f\|_{W^{p,q}_{s}} 
= \| \langle k \rangle^{s} \, \kappa(D-k)f(x)
\|_{\ell^{q}_{k}(\Z^n) L^p_{x}(\R^n)}
\end{equation*}
is finite. 
If 
$s=0$, we write 
$W^{p,q} = W^{p,q}_{0}$. 
The space $W^{p,q}_{s}$ 
does not depend on the choice of the function $\kappa$ 
up to the equivalence of quasi-norm. 
The space $W^{p,q}_{s}$ is a quasi-Banach space 
(Banach space if $1 \leq  p,q \leq \infty$) and 
$\mathcal{S} \subset W^{p,q}_{s} \subset \mathcal{S}^\prime$. 
If $0 < p,q < \infty$, then $\mathcal{S}$ is dense in $W^{p,q}_{s}$.
It is known that $W^{p,2}$ is equivalent to 
the $L^2$-based amalgam space $(L^2, \ell^{p})$
equipped with the quasi-norm
$\| f \|_{ (L^2, \ell^p) } 
= \| f(x+\nu)
\|_{ L^{2}_{x} (Q) \ell^{p}_{\nu} (\Z^n) }$.
See Feichtinger \cite{Feichtinger-1981} and Triebel \cite{Triebel-1983}
for more details. 
See also \cite{Fei-1983, GS-2004, Gro-book-2001, Kob-2006, WH-2007} 
for modulation spaces
which are variation of the Wiener amalgam spaces.
For amalgam spaces, see \cite{FS, Holland}.

Some of the relations between $W^{p,q}_{s}$ and the spaces 
$L^p$, $h^p$, and $bmo$ will be given below.  
\begin{lem} \label{Waembd}
Let $s \in \R$ and $0 < p, p_1, p_2, q_1, q_2 \le \infty$.
Then,
\begin{align}
&
W^{p_1,q_1}_{s} \hookrightarrow W^{p_2,q_2}_{s}
\quad\textrm{if}\quad
p_1\le p_2, \;\; q_1\le q_2; 
\label{emb-WW}
\\
&
h^p \hookrightarrow W^{p,2}_{ \alpha(p) }, 
\quad\text{where}\quad
\alpha(p) = n/2 - \max \{{n}/{2}, {n}/{p} \}; 
\label{emb-hW}
\\
&
bmo \hookrightarrow W^{\infty, 2} .
\label{emb-bmoW}
\end{align}
\end{lem}

\begin{proof}
The embedding \eqref{emb-hW}
is given in \cite[Theorems 1.1]{CKS-JFA} for $1 < p \le \infty$
and in \cite[Theorem 1.2]{GWYZ-JFA} for $0 < p \le 1$.
The explicit proofs of \eqref{emb-WW} and \eqref{emb-bmoW}
can be found in \cite[Lemma 2.2]{KMT-JFA}.
\end{proof}

\section{Main result}\label{secMainResult}

\subsection{Refined version of the main theorem}
In this subsection, 
we give a slight extension of 
Theorem \ref{main-thm-1}.
To do this, we shall define the following symbol class
which can be wider than the class 
stated in Definition \ref{def-main-thm-1}.
\begin{defn}
Let $N \ge 2$.
For $\vecm = (m_1,\dots,m_N) \in \R^N$,
$\vecs \in [0, \infty)^{N+1}$,
and $t \in (0, \infty]$,
we denote by 
$S^{\vecm}_{0,0} (\vecs, t ; \R^n, N)$
the set of all $\sigma \in L^{\infty} ((\R^n)^{N+1})$
such that the quasi-norm
\begin{align*} 
\|\sigma\|_{ S^{\vecm}_{0,0} (\vecs,t; \R^n, N) } 
=
\bigg\{
\sum_{ \veck \in (\N_0)^{N+1} }
\big( 2^{ \vecs\cdot \veck } \big)^{t} \,
\Big\|
\prod_{j=1}^{N} \langle \xi_j \rangle^{-m_j}
\Delta_{\veck} 
\sigma (x, \vecxi)
\Big\|_{ L^2_{ul,\vecxi} ((\R^{n})^{N}) L^{\infty}_{x} (\R^{n}) } ^{t}
\bigg\}^{1/t}
\end{align*}
is finite,
with a usual modification when $t=\infty$.
\end{defn}

The definition above does not depend on 
the choice of the Littlewood--Paley partition
up to the equivalence of quasi-norms.
Also, the same applies to the class 
given in Definition \ref{def-main-thm-1}.

Notice that 
if $m_1, \dots, m_N \le 0$ and
if $m_1+\dots+m_N = m$,
then 
$S^{m}_{0,0} (\vecs, t ; \R^n, N)
\subset 
S^{\vecm}_{0,0} (\vecs, t ; \R^n, N)$.
Then, we see that
the theorem below induce 
the statement of Theorem \ref{main-thm-1}. 

\begin{thm}\label{main-thm-2}
Let $N\ge 2$,
$0<p, p_1,\dots,p_N \le \infty$, and
$1/p \le 1/p_1 + \dots + 1/p_N$.
\begin{enumerate}
\item
Let $m_1, \dots, m_N \in \R$ satisfy
\begin{equation} \label{criticalvecm1}
- \max \Big\{ \frac{n}{p_j}, \frac{n}{2} \Big\} 
< m_j <
\frac{n}{\,2\,} - \max \Big\{ \frac{n}{\,p_j}, \frac{n}{2} \Big\},  
\quad j=1, \dots, N, 
\end{equation}
and
\begin{equation} \label{criticalvecm2}
m_1 + \dots +m_N = 
\min \Big\{ \frac{n}{p}, \frac{n}{2} \Big\} 
- \sum_{j =1}^{N} 
\max \Big\{ \frac{n}{p_j}, \frac{n}{2} \Big\} .
\end{equation}
\begin{enumerate}
\item
If $0 < p < 2$,
$s_0 = {n}/{2}$, and
$s_j = \max \{ {n}/{p_j}, {n}/{2} \}$,
$j=1,\dots,N$, 
then 
\[
\op\left( S^{\vecm}_{0,0} \big( \vecs,\min\{1,p\}; \R^n, N \big) \right) \subset
B(h^{p_1} \times \dots \times h^{p_N} \to h^{p}) .
\] 

\item
If $2 \le p < \infty$,
$s_0 = {n}/{p}$, and
$s_j = {n}/{2}$,
$j=1,\dots,N$,
then 
\[
\op\left( S^{\vecm}_{0,0} \big( \vecs, 1; \R^n, N \big) \right) \subset
B(h^{p_1} \times \dots \times h^{p_N} \to L^{p}) .
\] 
\end{enumerate}

\item
Let
$m= - \sum_{j =1}^{N} 
\max \{ {n}/{p_j}, {n}/{2} \}$.
If $s_0 = 0$ and $s_j = {n}/{2}$, $j=1,\dots,N$,
then 
\[
\op\left( S^{m}_{0,0} \big( \vecs ,1; \R^n, N \big) \right) \subset
B(h^{p_1} \times \dots \times h^{p_N} \to L^{\infty}) .
\] 
\end{enumerate}
In the above assertions, if some of the $p_j$'s, $j\in \{1, \dots, N\}$, 
are equal to $\infty$, 
then the conclusions  
hold with the corresponding $h^{p_j}$ replaced by $bmo$. 
\end{thm}

\begin{rem} \label{Remark-class}
The boundedness stated in Theorem \ref{main-thm-2} holds still true
even if the norm
\[
\| f(x,\vecxi) \|_{ L^2_{ul,\vecxi} ((\R^{n})^{N}) L^{\infty}_{x} (\R^{n}) }
\]
of the classes
$S^{\vecm}_{0,0} ( \vecs, t; \R^n, N )$ and 
$S^{m}_{0,0} ( \vecs, t; \R^n, N )$
is replaced by the better one
\[
\sup_{ \nu_0,\nu_1, \dots, \nu_N \in \Z^n }
\left\| f(x+\nu_0, \xi_1+\nu_1,\dots,\xi_N+\nu_N) 
\right\|_{ L^{2}_{(\xi_1,\dots,\xi_N)} (Q^N) L^{\max \{p,2\}}_{x} (Q) } .
\]
This can be seen by a careful following of 
the proof given in Section \ref{secProof},
but since its proof becomes much more complicated,
we leave it to the interested readers.
\end{rem}

\subsection{Key proposition}

Proposition \ref{main-prop} below
plays a crucial role in our argument
and contains the essential part of Theorem \ref{main-thm-2}.
The proof will be given in the succeeding sections.

\begin{prop} \label{main-prop}
Let $N\ge 2$,
$0<p, p_1,\dots, p_N \le \infty$, and
$1/p \le 1/p_1 + \dots + 1/p_N$.
Suppose $\sigma \in L^{\infty} ((\R^n)^{N+1})$ satisfies
$\supp \calF \sigma \subset 
B_{R_0} \times B_{R_1} \times \dots \times B_{R_N}$
for $R_{0}, R_{1}, \dots, R_{N} \in [1, \infty)$.
\begin{enumerate}
\item
Let $m_1, \dots, m_N \in \R$ satisfy 
\eqref{criticalvecm1} and \eqref{criticalvecm2}.
\begin{enumerate}
\item
If $0 < p < 2$,
then 
\begin{equation*}
\| T_{ \sigma } \|_{h^{p_1} \times \dots \times h^{p_N} \to h^{p}}
\lesssim
R_{0}^{n/2}
\prod_{j=1}^N R_{j}^{ \max\{ n/p_j , n/2 \} }
\Big\| \prod_{j=1}^{N}
\langle \xi_j \rangle^{-m_j} \sigma (x,\vecxi) 
\Big\|_{ L^2_{ul,\vecxi} ((\R^{n})^{N}) L^{\infty}_{x} (\R^{n})} .
\end{equation*}

\item
If $2 \le p < \infty$,
then 
\begin{equation*}
\| T_{ \sigma } \|_{h^{p_1} \times \dots \times h^{p_N} \to L^{p}}
\lesssim
R_{0}^{n/p}
\prod_{j=1}^N R_{j}^{ n/2 }
\Big\| \prod_{j=1}^{N}
\langle \xi_j \rangle^{-m_j} \sigma (x,\vecxi) 
\Big\|_{ L^2_{ul,\vecxi} ((\R^{n})^{N}) L^{\infty}_{x} (\R^{n})} .
\end{equation*}
\end{enumerate}

\item
Let 
$m = - \sum_{j =1}^{N} 
\max \{ {n}/{p_j}, {n}/{2} \}$.
Then,
\begin{equation*}
\| T_{ \sigma } \|_{h^{p_1} \times \dots \times h^{p_N} \to L^{\infty}}
\lesssim
\prod_{j=1}^N R_{j}^{n/2}
\| 
\langle \vecxi \rangle^{-m} \sigma (x,\vecxi) 
\|_{ L^2_{ul,\vecxi} ((\R^{n})^{N}) L^{\infty}_{x} (\R^{n})} . 
\end{equation*}
\end{enumerate}
In the above assertions, if some of the $p_j$'s, $j\in \{1, \dots, N\}$, 
are equal to $\infty$, 
then the conclusions  
hold with the corresponding $h^{p_j}$ replaced by $bmo$. 
\end{prop}

\subsection{Proof of Theorem \ref{main-thm-2}} 

Boulkhemair \cite{Boulkhemair} first pointed out that,
in order to investigate the smoothness condition
to assure the $L^2$-boundedness of the linear pseudo-differential operators,
it suffices to consider the boundedness for symbols
whose Fourier supports are compact.
Our strategy relies heavily on his idea.
We shall proceed to the proof of Theorem \ref{main-thm-2}.
We decompose the symbol $\sigma$ into
the sum of
$\Delta_{\veck} \sigma$
over $\veck \in (\N_0)^{N+1}$.
Since the 
support of 
$\calF (\Delta_{\veck} \sigma)$ 
is included in 
$B_{R_0} \times B_{R_1} \times \cdots \times B_{R_N}$
with 
$R_{j}= 2^{ k_{j} +1}$,
$j=0,1,\dots,N$,
we see that Theorem \ref{main-thm-2} follows 
by applying Proposition \ref{main-prop} to 
the decomposed operators $T_{\Delta_{\veck} \sigma}$.

\subsection{Symbols with classical derivatives} 
\label{secClassicalDerivatives} 
The following proposition shows  
that symbols that have classical derivatives 
up to certain order satisfy the conditions of 
Theorem \ref{main-thm-2}. 

\begin{prop} \label{classicalderivative} 
Let $N \ge 2$, 
$m, m_1, \dots, m_N \in \R$,
$s_0, s_1, \dots, s_N \in [0, \infty)$,
and $t \in (0, \infty]$. 
If a bounded measurable 
function $\sigma$ on $(\R^n)^{N+1}$
satisfies
\[
 | \partial_{x}^{ \alpha_{0} }
\partial_{ \xi_{1} }^{ \alpha_{1} } \cdots
\partial_{ \xi_{N} }^{ \alpha_{N} }
\sigma (x, \xi_1, \dots, \xi_N) |
\le 
(1+|\xi_1|+\dots+|\xi_N|)^{m}
\]
or
\[
 | \partial_{x}^{ \alpha_{0} }
\partial_{ \xi_{1} }^{ \alpha_{1} } \cdots
\partial_{ \xi_{N} }^{ \alpha_{N} }
\sigma (x, \xi_1, \dots, \xi_N) |
\le 
\prod_{j=1}^{N} (1+|\xi_j|)^{m_j} 
\]
for
$\alpha_{j}\in (\N_0)^n$ 
with 
$|\alpha_{j}| \le [s_{j}]+1$,
where 
$[s_j]$ is the integer part of $s_j$,
then
$\sigma \in 
S^{m}_{0,0} ( \vecs, t ; \R^n, N)$
or
$\sigma \in 
S^{\vecm}_{0,0} ( \vecs, t ; \R^n, N)$,
respectively.

To be precise, 
the above assumptions should be understood that 
the derivatives of $\sigma$ taken 
in the sense of distribution 
are locally integrable functions on $(\R^{n})^{N+1}$ 
and they satisfies the inequality stated above almost everywhere.  
\end{prop}

Since statements quite similar to Proposition \ref{classicalderivative} 
are already proved in
\cite[Proposition 4.7]{KMT-JPDOA}
and \cite[Proposition 5.4]{KMT-JMSJ}, 
we omit the proof here.

\section{Lemmas for the proof of Proposition \ref{main-prop}}
\label{seclemmasforProof}

In this section, we
collect some lemmas to prove Proposition \ref{main-prop}. 
The following will be used to decompose symbols,
which was essentially proved in \cite[Lemma 2.2.1]{Sugimoto-JMSJ}.
The explicit proof can be found in \cite[Lemma 4.4]{KMT-JMSJ}.

\begin{lem} \label{unifdecom}
There exist functions $\kappa \in \calS(\R^n)$ and $\chi \in \calS(\R^n)$ 
such that
$\supp \kappa \subset [-1,1]^n$, $\supp \widehat \chi \subset B(0,1)$, 
$| \chi | \geq c > 0$ on $[-1,1]^n$
and
\begin{equation*}
\sum_{\nu\in\Z^n} \kappa (\xi - \nu) \chi (\xi - \nu ) = 1, \quad \xi \in \R^n.
\end{equation*}
\end{lem}

The two lemmas below
play important roles to obtain the boundedness 
for the multilinear H\"ormander class with the critical order
(Theorem B) in \cite{KMT-JFA}.
We will again use them in the present paper.
See \cite[Lemmas 2.4 and 2.5]{KMT-JFA}
for these proofs.

\begin{lem} \label{productLweakp'}
Let $N \ge 2$, $1<r<\infty$, 
and let $a_1, \dots, a_N \in \R$ satisfy
$-n/2< a_j < 0$ and 
$\sum_{j=1}^N a_j=n/r-Nn/2$. 
Then the following holds for all nonnegative functions 
$A_1, \dots, A_N$ on $\Z^n$: 
\[
\sum_{ \nu_1, \, \dots, \, \nu_N \in \Z^n }
\, 
A_0(\nu_1+\dots+ \nu_N)
\,
\prod_{j=1}^N (1+ |\nu_j| )^{a_{j}} A_j(\nu_j)
\lesssim 
\| A_0 \|_{ \ell^{r} (\Z^n) }
\prod_{j=1}^N
\|A_j\|_{\ell^2 (\Z^n) }.  
\]
\end{lem}

\begin{lem} \label{productLweak1}
Let $N \ge 2$. 
Then the following holds for all nonnegative functions 
$A_1, \dots, A_N$ on $\Z^n$: 
\begin{equation*}
\sum_{
\nu_1, \,\dots, \,\nu_N \in \Z^n}
(1+ |\nu_1| + \dots + |\nu_N| )^{-Nn/2} 
\prod_{j=1}^N 
A_j(\nu_j)
\lesssim \prod_{j=1}^N
\|A_j\|_{\ell^2 (\Z^n) } . 
\end{equation*}
\end{lem}

For $0 < r < \infty$, 
we denote by $S_r$ the operator 
\begin{equation*}
S_{r} (f) (x) 
= \Big( \int_{\R^n} 
\frac{ |f(x-z)|^{r} }{ \langle z \rangle^{n+1} } 
\, dz \Big)^{1/r}
\end{equation*}
for $f \in \calS(\R^n)$.
Obviously, $S_{r}$ is bounded on $L^p$ for $p \ge r$.
The lemma below was
proved in \cite[Lemma 4.1]{KMT-JMSJ} 
for the case $r=2$. 
We extend it to the general case $0 < r < \infty$.

\begin{lem} \label{unidec->S}
Let $0 < r < \infty$ and 
$\kappa \in \calS(\R^n)$ satisfy
$\supp \kappa \subset [-1,1]^n$.
Then
\begin{equation} \label{est_unidec->S}
\left| \kappa ( D-\nu ) f(x) \right|
\lesssim
S_r ( \kappa ( D-\nu ) f ) (y)
\end{equation}
holds for any $f \in \calS(\R^n)$, 
$\nu \in \Z^n$, and 
$x, y \in \R^n$ satisfying $|x-y|\lesssim1$.
\end{lem}

\begin{proof}
Taking $\varphi \in \calS(\R^n)$ satisfying 
$\varphi=1$ on $[-1,1]^n$ and 
$\supp \varphi \subset [-2,2]^n$,
we have
\begin{align*}&
\kappa ( D-\nu ) f(x)
= 
\varphi ( D-\nu )\kappa ( D-\nu ) f(x) .
\end{align*}
If $0 < r \le1$,
Nikol'skij's inequality gives
\begin{align*}
| \kappa ( D-\nu ) f(x) |
&\le
\left\| 
(\calF^{-1} \varphi) (z) \kappa ( D-\nu ) f(x-z)
\right\|_{L^{1}_{z} (\R^n)} 
\\& \lesssim
\left\| 
(\calF^{-1} \varphi) (z) \kappa ( D-\nu ) f(x-z) 
\right\|_{L^{r}_{z} (\R^n)} 
\lesssim
S_r( \kappa ( D-\nu ) f ) (x).
\end{align*}
The implicit constants above depend
only on $r$, $n$, and 
the diameters of $\supp \varphi$ and $\supp \kappa$.
Since
$S_r( f ) (x) \approx S_r( f ) (y)$
for $|x-y|\lesssim1$,
we obtain \eqref{est_unidec->S} for $0 < r \le 1$.
Using H\"older's inequality and repeating the statement above,
we obtain \eqref{est_unidec->S} for $1 < r < \infty$.
\end{proof}

The lemma below might be well-known
(see \cite[Subsection 2.3]{KMT-JMSJ}).
Also, the readers familiar with Wiener amalgam spaces
may realize that the following inequality is immediately deduced from
the embedding $W^{p,2} \hookrightarrow h^{p}$, $0 < p \le 2$,
proved in \cite[Theorem 1.2]{GWYZ-JFA}.

\begin{lem} \label{Amalgam-Hardy}
Let $0 < p \le 2$.
If $g \in \calS(\R^n)$ satisfies that
$| g | \geq c > 0$ on $[-1,1]^n$
with some positive constant $c$,
then 
\[
\| f \|_{h^p (\R^n) }
\lesssim
\| V_{g} f (x,\xi)
\|_{L^{2}_{\xi} (\R^n) L^{p}_{x} (\R^n)} .
\]
\end{lem}

\begin{proof}
Since $(L^2,\ell^p) \hookrightarrow h^p$ for $0 < p \leq 2$
(see \cite[Section 2.3]{KMT-JMSJ}),
it suffices to prove
\begin{equation}
\label{amalgam-Wiener}
\| f \|_{ (L^2, \ell^p) }
\lesssim
\| V_{g} f (x,\xi)
\|_{L^{2}_{\xi} (\R^n) L^{p}_{x} (\R^n)} .
\end{equation}
Since $|g(x-y)| \ge c$ for $x,y \in Q$,
it holds that
\begin{align*}
\| f \|_{ (L^2, \ell^p) } 
=
\| f(x+\nu)
\|_{ L^{2}_{x} (Q) \ell^{p}_{\nu} (\Z^n) }
\lesssim
\| \overline {g (x-y)}
f(x+\nu)
\|_{ L^{2}_{x} (Q) \ell^{p}_{\nu} (\Z^n)}
\end{align*}
for any $y \in Q$,
which implies from 
the embedding $L^{2} (\R^n) \hookrightarrow L^{2} (Q)$
that
\begin{align*}
\| f \|_{ (L^2, \ell^p) } 
\lesssim
\| \overline {g (x-y)}
f(x+\nu)
\|_{ L^{2}_{x} (Q) \ell^{p}_{\nu} (\Z^n) L^{p}_{y}(Q)}
\le
\| \overline {g (x-y)}
f(x+\nu)
\|_{ L^{2}_{x} (\R^n) \ell^{p}_{\nu} (\Z^n) L^{p}_{y}(Q)} .
\end{align*}
By recalling the definition of $V_{g}$ stated in \eqref{STFT},
the last quantity is identical with
\begin{align*}
\| \overline {g (x-y)}
f(x)
\|_{ L^{2}_{x} (\R^n) L^{p}_{y}(\R^n)} 
=
\| V_{g} f (y,\xi)
\|_{L^{2}_{\xi} (\R^n) L^{p}_{y} (\R^n)} .
\end{align*}
This completes the proof of \eqref{amalgam-Wiener}.
Here, note that the opposite inequality of \eqref{amalgam-Wiener} holds.
However, since the equivalence is unnecessary here, we omit the detail.
\end{proof}

The following lemma
was already given in \cite[Lemma 3.2]{MT-AIF} 
for the case $p=2$ and $R=1$.
We extend it to a bit more general form.
Moreover, we remark that
the inequality below implies the embedding 
$L^p \hookrightarrow W^{p, p'}$, $1 \le p \le 2$,
proved in \cite[Theorem 1.1]{CKS-JFA}.

\begin{lem} \label{MT-lemma24}
Let $2 \le p \le \infty$,
$R \ge 1$,
and $\varphi \in \calS(\R^n)$. 
Then 
\[
\Big\| \varphi \Big( \frac{D-\nu}{R} \Big) f(x)
\Big\|_{\ell^{p}_{\nu} (\Z^n) L^{p'}_{x} (\R^n)} 
\lesssim
R^{n/p} \|f\|_{L^{p'} (\R^n) }.
\]
\end{lem}

\begin{proof}
With the notation $\Phi=\calF^{-1} \varphi$,
the expression \eqref{cubicdiscretization} with $\ell = 2\pi$
yields that
\begin{align*}
&
\varphi \Big( \frac{D-\nu}{R} \Big) f(x)
=
R^n \int_{\R^n} 
e^{ i y \cdot \nu } \,
\Phi (Ry) f(x - y) 
\, dy
\\&=
R^n
\int_{ 2\pi Q} e^{ i y \cdot \nu } 
\bigg\{ 
\sum_{\nu' \in \Z^n}
\Phi \big( R (y + 2\pi\nu' ) \big) f( x-y-2\pi\nu' ) 
\bigg\} 
\, dy .
\end{align*}
We realize that the function 
$\sum_{\nu' \in \Z^n} \Phi (R(y+2\pi\nu')) f( x-y-2\pi\nu' ) $
is $2\pi\Z^n$-periodic with respect to the $y$-variable.
Hence, we have by Hausdorff--Young's inequality
\begin{align*}
\Big\| \varphi \Big( \frac{D-\nu}{R} \Big) f(x)
\Big\|_{ \ell^{p}_{\nu} }^{p'}
\lesssim R^{np'} 
\int_{2 \pi Q}
\bigg| 
\sum_{\nu' \in \Z^n} 
\Phi \left( R ( y + 2\pi\nu' ) \right) 
f( x-y-2\pi\nu' ) 
\bigg|^{p'} \,dy .
\end{align*}
Since
$\sum_{\nu' \in \Z^n} | \Phi (R(y+2\pi\nu')) |
\lesssim 1$
for any $y\in\R^n$ and $R \geq 1$,
by applying H\"older's inequality to the sum over $\nu'$, 
the integral of the right hand side is bounded by
\begin{align*}
&
\int_{2 \pi Q} \sum_{\nu' \in \Z^n}
\left| \Phi \left( R ( y+2\pi\nu' ) \right) \right| 
\left| f( x-y-2\pi\nu' ) \right|^{p'} \,dy
= 
\big\| |\Phi (Ry)|
|f(x-y)|^{p'} 
\big\|_{ L^{1}_{y} } ,
\end{align*}
where we again used \eqref{cubicdiscretization} in the identity above.
Therefore, we obtain
\begin{align*}
\Big\| \varphi \Big( \frac{D-\nu}{R} \Big) f(x)
\Big\|_{ \ell^{p}_{\nu} }^{p'}
\lesssim R^{np'} 
\big\|
| \Phi ( R y ) | 
| f( x-y ) |^{p'} 
\big\|_{ L^{1}_{y} } .
\end{align*}
Integrating over $x$, we have
\begin{align*}
\Big\| \varphi \Big( \frac{D-\nu}{R} \Big) f(x)
\Big\|_{ \ell^{p}_{\nu} L^{p'}_{x} }^{p'}
&\lesssim R^{np'} 
\big\| |\Phi (Ry)|
|f(x-y)|^{p'} 
\big\|_{ L^{1}_{y} L^{1}_{x} }
\approx
R^{n(p'-1)} \|f\|_{L^{p'}}^{p'} ,
\end{align*}
which completes the proof.
\end{proof}

\section{Proof of Proposition \ref{main-prop}}
\label{secProof}

In this section, we will use the following notation:
$\vecxi=(\xi_1,\dots,\xi_N) \in (\R^n)^N$,
$\vecnu=(\nu_1,\dots,\nu_N) \in (\Z^n)^N$, and
$d\vecxi = d\xi_1 \dots d\xi_N$.
Also, we remark that, 
for any $p, p_j \in (0, \infty]$ satisfying 
$1/p \le 1/p_1 + \dots + 1/p_N$, 
we can choose
$\widetilde{p}_{j} \in (0, \infty]$, $j = 1, \dots, N$, 
such that
\begin{align} \label{pjs}
\begin{split}
\frac{1}{p} = \frac{1}{ \widetilde{p}_1 } + \dots + \frac{1}{ \widetilde{p}_N }
\quad \textrm{and}\quad
p_j \le \widetilde{p}_j .
\end{split}
\end{align}
In fact, for instance, 
if $p=\infty$, then we can take $\widetilde{p}_{j} = \infty$,
and if $p<\infty$, then we can take
\[
\frac{1}{ \widetilde{p}_j } = 
\frac{1}{p} \, \frac{1}{p_j} 
\Big( \frac{1}{p_1}+\dots+\frac{1}{p_N} \Big)^{-1} .
\]

Now, we shall give a proof of Proposition \ref{main-prop}.
First, we decompose $T_{\sigma}$ as follows.
By Lemma \ref{unifdecom},
the symbol $\sigma$ can be written as
\begin{align*}
\sigma ( x, \vecxi)
=
\sum_{\vecnu\in(\Z^n)^N} 
\sigma ( x, \vecxi ) 
\prod_{j=1}^N \kappa (\xi_{j} - \nu_{j}) \chi (\xi_{j} - \nu_{j} )
=
\sum_{\vecnu\in(\Z^n)^N} 
\sigma_{\vecnu} ( x, \vecxi ) 
\prod_{j=1}^N \kappa (\xi_{j} - \nu_{j})
\end{align*}
with
\begin{equation*}
\sigma_{\vecnu} ( x, \vecxi ) 
=
\sigma ( x, \vecxi ) 
\prod_{j=1}^N \chi (\xi_{j} - \nu_{j} ) .
\end{equation*}
Then, by denoting
the operators $\kappa (D - \nu_{j})$ 
by $\square_{\nu_{j}}$, $j=1,\dots,N$, 
we can write as
\begin{align} \label{decomposedT}
\begin{split}
T_{ \sigma }( f_1, \dots, f_N )(x) &=
\sum_{\vecnu\in(\Z^n)^N} 
T_{ \sigma_{\vecnu} }( \kappa(D-\nu_1)f_1, \dots, \kappa(D-\nu_N) f_N )(x) 
\\&=
\sum_{\vecnu\in(\Z^n)^N} 
T_{ \sigma_{\vecnu} }( \square_{\nu_1}f_1, \dots, \square_{\nu_N}f_N )(x) .
\end{split}
\end{align}
Here we remark that, 
for $0 < p, q \le \infty$ and $s \in \R$,
it holds that
$\| \langle \nu \rangle^{s} \square_{\nu} f \|_{\ell^{q} L^{p}} 
\lesssim \| f \|_{W^{p, q}_{s}}$,
since $\supp \kappa$ is compact.
Now, this 
$T_{ \sigma_{\vecnu } }
( \square_{\nu_1}f_1, \dots, \square_{\nu_N}f_N )$
satisfies the following inequality.

\begin{lem}
\label{main-lemma}
Let $N\ge 2$.
For $j = 1,\dots,N$,
let $m, m_{j} \in (-\infty, 0]$, 
$r_{j} \in (0,\infty)$, and
$R_{0}, R_{j} \in [1, \infty)$.
Suppose $\sigma$ is a bounded continuous function 
on $(\R^n)^{N+1}$ satisfying
$\supp \calF \sigma \subset 
B_{R_0} \times B_{R_1} \times \dots \times B_{R_N}$
and write
$W(\vecxi) = \langle \vecxi \rangle^{m}$
or
$\prod_{j=1}^{N} \langle \xi_j \rangle^{m_j}$.
Then, 
\begin{align*}&
\left| T_{ \sigma_{\vecnu} }
( \square_{\nu_1}f_1, \dots, \square_{\nu_N}f_N )(x) \right|
\\&\lesssim
W(\vecnu)
\| W(\vecxi)^{-1}
\sigma (x,\vecxi) \|_{ L^2_{ul,\vecxi} ((\R^{n})^{N}) L^{\infty}_{x} (\R^{n})}
\prod_{j=1}^{N}
\left\|
S_{r_j}( \square_{\nu_{j}} f_{j} ) (y+\tau_{j}) 
\right\|_{ \ell^{2}_{\tau_{j}} ( \Lambda_{R_{j}} ) } 
\end{align*}
holds for any  $x, y \in \R^n$ satisfying $|x-y|\lesssim1$,
where $\Lambda_{R_{j}} = \Z^n \cap [-2R_{j}-1, 2R_{j}+1]^n$.
\end{lem}

\begin{proof}
Since
the support of 
$(\calF_{1,\dots,N} \sigma_{\vecnu} ) (x, \cdot)$
is included in 
$B_{2R_1} \times \dots \times B_{2R_N}$ 
for any $x \in \R^n$,
\begin{align*} 
T&:=
T_{ \sigma_{\vecnu} }
( \square_{\nu_1}f_1, \dots, \square_{\nu_N}f_N )(x) 
= \frac{1}{(2\pi)^{Nn}}
\int_{(\R^{n})^{N}}
\big( \calF_{1,\dots,N} \sigma_{\vecnu} \big) (x, \vecz)
\, \prod_{j=1}^N \square_{\nu_{j}} f_{j} (x+z_{j}) 
\, d\vecz
\\&= \frac{1}{(2\pi)^{Nn}}
\int_{(\R^{n})^{N}}
\big( \calF_{1,\dots,N} \sigma_{\vecnu} \big) (x, \vecz)
\prod_{j=1}^N \ichi_{B_{2R_{j}}}(z_{j}) \;
\square_{\nu_{j}} f_{j} (x+z_{j}) 
\, d\vecz .
\end{align*}
Since 
the ball $B_{2R_j}$ is covered by a disjoint union of the unit cubes
$\tau + Q$, $\tau \in \Lambda_{R_j}$, 
the characteristic function $\ichi_{B_{2R_j}}$ is bounded by
the sum of $\ichi_{Q}(\cdot-\tau)$ over $\tau \in \Lambda_{R_j}$.
This yields
\begin{align} \label{Tsigma1}
\begin{split}&
|T|
\le
\int_{(\R^{n})^{N}}
\left| ( \calF_{1,\dots,N} \sigma_{\vecnu} ) (x, \vecz) \right|
\prod_{j=1}^N \ichi_{B_{2R_{j}}}(z_{j})
\left| \square_{\nu_{j}} f_{j} (x+z_{j}) \right| 
\, d\vecz
\\&\le
\sum_{ \tau_1 \in \Lambda_{R_1} } \dots \sum_{ \tau_N \in \Lambda_{R_N} }
\int_{(\R^{n})^{N}}
\left| ( \calF_{1,\dots,N} \sigma_{\vecnu} ) (x, \vecz) \right|\prod_{j=1}^N \ichi_{Q}(z_{j}-\tau_{j})
\left| \square_{\nu_{j}} f_{j} (x+z_{j}) \right|
\, d\vecz .
\end{split}
\end{align}
Note that $|(x+z_j)-(y+\tau_j)| \lesssim 1$ if $|x-y| \lesssim 1$ and $z_j - \tau_j \in Q$.
Then, by Lemma \ref{unidec->S} and the Cauchy--Schwarz inequality, 
the integral above is estimated by
\begin{align} \label{Tsigma2}
\begin{split}
&
\prod_{j=1}^N 
S_{r_j} ( \square_{\nu_{j}} f_{j} ) (y+\tau_j)
\int_{(\R^{n})^{N}}
\left| ( \calF_{1,\dots,N} \sigma_{\vecnu} ) (x, \vecz) \right|
\prod_{j=1}^N \ichi_{Q}(z_{j}-\tau_{j})
\, d\vecz
\\&\le
\prod_{j=1}^N 
S_{r_j} ( \square_{\nu_{j}} f_{j} ) (y+\tau_j)
\Big\| ( \calF_{1,\dots,N} \sigma_{\vecnu} ) (x, \vecz) 
\prod_{j=1}^N \ichi_{Q}(z_{j}-\tau_{j})
\Big\|_{L^{2}_{\vecz}} .
\end{split}
\end{align}
Combining with \eqref{Tsigma1} and \eqref{Tsigma2}
and then using the Cauchy--Schwarz inequalities to 
the sums with respect to the $\tau_j$'s, $j=1,\dots,N$,
we have
\begin{align} \label{Tsigma3}
\begin{split}
|T|
&\leq
\prod_{j=1}^N 
\left\| S_{r_j} ( \square_{\nu_{j}} f_{j} ) (y+\tau_j) \right\|_{ \ell^{2}_{\tau_j} (\Lambda_{R_{j}}) }
\Big\| ( \calF_{1,\dots,N} \sigma_{\vecnu} ) (x, \vecz) 
\prod_{j=1}^N \ichi_{Q}(z_{j}-\tau_{j})
\Big\|_{ L^2_{\vecz} \ell^{2}_{\vectau} } 
\\&=
\prod_{j=1}^N 
\left\| S_{r_j} ( \square_{\nu_{j}} f_{j} ) (y+\tau_j) 
\right\|_{ \ell^{2}_{\tau_j} (\Lambda_{R_{j}}) }
\| \sigma_{\vecnu} (x, \vecxi) 
\|_{ L^2_{\vecxi} } ,
\end{split}
\end{align}
where, in the identity, 
we applied
\eqref{cubicdiscretization}
to the $L^2_{\vecz} \ell^{2}_{\vectau}$ norm
and then used Plancherel's theorem.
In what follows, we shall prove that
\begin{equation} \label{Tsigma4}
\sup_{x \in \R^n} 
\| \sigma_{\vecnu} (x, \vecxi) 
\|_{ L^2_{\vecxi} ((\R^n)^{N}) }
\lesssim
W(\vecnu)
\| W(\vecxi)^{-1} \sigma (x, \vecxi) 
\|_{ { L^2_{ul, \vecxi} ((\R^n)^{N}) } L^{\infty}_{x} (\R^n) }
\end{equation}
for $\vecnu \in (\Z^n)^N$.
Then, \eqref{Tsigma3} and \eqref{Tsigma4}
give the estimate of this lemma.
By \eqref{cubicdiscretization},
we have
\begin{align*}&
\| \sigma_{\vecnu} (\cdot, \vecxi) 
\|_{ L^2_{\vecxi} } 
= W(\vecnu)
\Big\| W(\vecnu)^{-1} \sigma ( \cdot, \vecxi+\vecmu ) 
\prod_{j=1}^N \chi (\xi_{j} +\mu_{j}- \nu_{j} ) 
\Big\|_{ L^2_{\vecxi} (Q^{N}) \ell^2_{\vecmu} ((\Z^n)^{N}) }
\\&\lesssim
W(\vecnu)
\Big\| W(\vecxi+\vecmu)^{-1} W(\vecxi+\vecmu-\vecnu)^{-1} \;
\sigma ( \cdot, \vecxi+\vecmu ) 
\prod_{j=1}^N \chi (\xi_{j} +\mu_{j}- \nu_{j} ) 
\Big\|_{ L^2_{\vecxi} (Q^{N}) \ell^2_{\vecmu} ((\Z^n)^{N}) } .
\end{align*}
Here the inequality above holds true since $m, m_{j} \in (-\infty,0]$ are assumed.
Since $\chi \in \calS(\R^n)$,
for some sufficiently large number $L>0$,
the $L^2\ell^2$-norm above is bounded by
\begin{align*} 
\Big\| 
W(\vecxi+\vecmu)^{-1}\sigma ( \cdot, \vecxi+\vecmu) \;
W(\vecmu-\vecnu)^{-1} 
\prod_{j=1}^N \langle \mu_{j} - \nu_{j} \rangle^{-L} 
\Big\|_{ L^2_{\vecxi} (Q^{N}) \ell^2_{\vecmu} } 
\lesssim 
\left\| W(\vecxi)^{-1}
\sigma (\cdot, \vecxi) \right\|_{ L^2_{ul, \vecxi} } . 
\end{align*}
Therefore, we obtain \eqref{Tsigma4},
and also complete the proof.\end{proof}


Now, we shall proceed to the estimates of 
the operators considered in Proposition \ref{main-prop}. 
In order to simplify the notations 
appearing in Lemma \ref{main-lemma},
let us denote
\begin{align}
& \label{sigmam}
| \sigma |_{m} =
\| \langle \vecxi \rangle^{-m} \sigma (x,\vecxi) 
\|_{ L^2_{ul,\vecxi} ((\R^{n})^{N}) L^{\infty}_{x} (\R^{n}) } ,
\\ 
& \label{sigmavecm}
| \sigma |_{ \vecm } =
\Big\| 
\prod_{j=1}^{N}
\langle \xi_j \rangle^{-m_j} \sigma (x,\vecxi) 
\Big\|_{ L^2_{ul,\vecxi} ((\R^{n})^{N}) L^{\infty}_{x} (\R^{n}) } ,
\end{align}
and further for $0 < p_j \le \infty$ and $0 < r_j < \infty$
\begin{align}
\label{tildeFnuj}
{F}_{\nu_{j}}^{ p_j, r_j } (x) =
S_{r_j} \big( \langle \nu_j \rangle^{ \alpha(p_j) } \square_{\nu_{j}} f_{j} \big) (x),
\quad \alpha(p_j) = n/2 - \max\{ n/2, n/p_j \} .
\end{align}

\subsection{Proof of Proposition \ref{main-prop} (1)-(a)}

Take a real valued function $g \in \calS(\R^n)$
satisfying 
$|g| \geq c >0$ on $[-1,1]^n$ and
$\supp \widehat g \subset B_1$.
We have by Lemma \ref{Amalgam-Hardy}
and duality
\begin{align} \label{hajime}
\begin{split}&
\| T_{ \sigma }
( f_1, \dots, f_N ) \|_{h^p}
\lesssim
\| V_{g} [ T_{ \sigma } ( f_1, \dots, f_N ) ] (x,\zeta)
\|_{L^{2}_{\zeta} (\R^n) L^{p}_{x} (\R^n)}
\\&=
\bigg\| \sup_{ h \in L^2 (\R^n) }
\Big| \int_{\R^n}
V_{g} [ T_{ \sigma } ( f_1, \dots, f_N ) ] (x,\zeta) \,
h(\zeta) \, d\zeta \Big|
\bigg\|_{L^{p}_{x} (\R^n)} .
\end{split}
\end{align}
Hence, in what follows, we consider
\begin{equation} \label{Idef}
I := \int_{\R^n}
V_{g} [ T_{ \sigma } ( f_1, \dots, f_N )] (x,\zeta) \,
h(\zeta) \, d\zeta
\end{equation}
for $x\in\R^n$ and $h\in L^2(\R^n)$, 
which is decomposed by \eqref{decomposedT} as
\begin{align*} 
I=
\sum_{\vecnu\in(\Z^n)^N} 
\int_{\R^n} 
V_{g} [ T_{ \sigma_{\vecnu} } ( \square_{\nu_1}f_1, \dots, \square_{\nu_N}f_N ) ] (x,\zeta) \,
h(\zeta) \, d\zeta .
\end{align*}
Here, we shall observe that
\begin{equation} \label{suppSTFTTsigma}
\supp 
V_{g} [ T_{ \sigma_{\vecnu} } ( \square_{\nu_1}f_1, \dots, \square_{\nu_N}f_N ) ] (x, \cdot)
\subset 
\big\{ \zeta \in \R^n : 
| \zeta-(\nu_1 +\dots+ \nu_N) | \lesssim R_0 \big\}.
\end{equation}
In fact, since 
$\supp \calF_{0} \sigma_{\vecnu} 
(\cdot, \vecxi ) 
\subset B_{R_0}$
and 
$\supp \kappa(\cdot-\nu_{j}) 
\subset \nu_{j}+[-1,1]^n$,
the identity
\begin{align*}
&
\calF [ T_{ \sigma_{\vecnu} }( \square_{\nu_1}f_1, \dots, \square_{\nu_N}f_N ) ] (\zeta)
\\&= \frac{1}{(2\pi)^{Nn}}
\int_{ (\R^{n})^N }
\big( \calF_0 \sigma_{\vecnu} \big) \big( \zeta - (\xi_1 +\dots+ \xi_N) , \vecxi \big) 
\, \prod_{j=1}^N \kappa (\xi_{j}-\nu_{j}) \widehat{f_{j}}(\xi_{j})
\, d\vecxi 
\end{align*}
implies that
\begin{equation} \label{suppFT}
\supp 
\calF [ T_{ \sigma_{\vecnu} }( \square_{\nu_1}f_1, \dots, \square_{\nu_N}f_N ) ]
\subset 
\big\{ \zeta \in \R^n : | \zeta-(\nu_1 +\dots+ \nu_N) | \lesssim R_0 \big\}.
\end{equation}
Hence, regarding the short-time Fourier transform 
given in \eqref{STFT} as
\[
V_{g} [ T_{ \sigma_{\vecnu} } ( \square_{\nu_1}f_1, \dots, \square_{\nu_N}f_N ) ] (x,\zeta)
=
\calF [ g (\cdot-x) T_{ \sigma_{\vecnu} }( \square_{\nu_1}f_1, \dots, \square_{\nu_N}f_N ) ] (\zeta) ,
\]
we see that \eqref{suppSTFTTsigma} holds.
Now, we take a function 
$\varphi \in \calS(\R^n)$ satisfying 
$\varphi = 1$ on 
$\{ \zeta\in\R^n: |\zeta|\lesssim 1\}$.
Then, the expression $I$ considered in \eqref{Idef} 
can be written as
\begin{align*}&
I
=
\sum_{\vecnu\in(\Z^n)^N} 
\int_{\R^n} 
V_{g} [ T_{ \sigma_{\vecnu} } ( \square_{\nu_1}f_1, \dots, \square_{\nu_N}f_N ) ] (x,\zeta) \,
\varphi \Big( \frac{\zeta - (\nu_1 +\dots+ \nu_N) }{R_0} \Big) 
h(\zeta) 
\, d\zeta 
\\&=
\sum_{\vecnu\in(\Z^n)^N} 
\int_{\R^n} 
g(t) \;
T_{ \sigma_{\vecnu} } ( \square_{\nu_1}f_1, \dots, \square_{\nu_N}f_N ) (x+t) \,
\calF \Big[ \varphi \Big( \frac{\cdot - (\nu_1 +\dots+ \nu_N) }{R_0} \Big) h \Big] (x+t) 
\, dt .
\end{align*}
By \eqref{cubicdiscretization},
we can further rewrite the above as
\begin{align} \label{Iest}
\begin{split}
I =
\sum_{\mu \in \Z^n}
\sum_{\vecnu\in(\Z^n)^N} 
\int_{Q} 
&
g(\mu+t) 
\;
T_{ \sigma_{\vecnu} }
( \square_{\nu_1}f_1, \dots, \square_{\nu_N}f_N ) (x+\mu+t)
\\& \times
\calF \Big[
\varphi \Big( \frac{\cdot - (\nu_1 +\dots+ \nu_N) }{R_0} \Big) 
h \Big] (x+\mu+t)
\, dt .
\end{split}
\end{align}

Now, we shall actually estimate the expression in \eqref{Iest}.
In this subsection, we will use Lemma \ref{main-lemma} with
$W(\vecxi) = \prod_{j=1}^{N} \langle \xi_j \rangle^{m_j}$.
By the fact that, for sufficiently large $L > 0$, 
$|g(\mu+t)| \lesssim \langle \mu \rangle^{-L}$ holds for $t \in Q$,
Lemma \ref{main-lemma} with the notation \eqref{sigmavecm} 
gives that
\begin{align*}
|I| &\lesssim
| \sigma |_{ \vecm }
\sum_{\mu \in \Z^n}
\langle \mu \rangle^{-L} 
\sum_{\vecnu\in(\Z^n)^N} 
\prod_{j=1}^{N}
\langle \nu_j \rangle^{m_j}
\left\|
S_{r_j} ( \square_{\nu_{j}} f_{j} ) (x+\mu+\tau_{j}) 
\right\|_{ \ell^{2}_{\tau_{j}} ( \Lambda_{R_{j}} ) } 
\\&\qquad \times
\int_{Q} 
\Big| \calF \Big[
\varphi \Big( \frac{\cdot - (\nu_1 +\dots+ \nu_N) }{R_0} \Big) 
h \Big] (x+\mu+t) \Big|
\, dt .
\end{align*}
In the right, we use the embedding
$L^2 (\R^n) \hookrightarrow L^1(Q)$
to the integral over $t \in Q$
and also,
by adding
$\langle \nu_j \rangle^{ -\alpha(p_j) } \langle \nu_j \rangle^{ \alpha(p_j) }$,
replace the factor $S_{r_j} ( \square_{\nu_{j}} f_{j} )$
with ${F}_{\nu_{j}}^{p_j,r_j}$ (see \eqref{tildeFnuj}).
Then 
\begin{align*} &
|I| 
\lesssim
| \sigma |_{ \vecm }
\sum_{\mu \in \Z^n}
\langle \mu \rangle^{-L} 
\\& \times
\sum_{\vecnu\in(\Z^n)^N} 
\prod_{j=1}^{N}
\langle \nu_j \rangle^{m_j - \alpha(p_j)}
\big\| {F}_{\nu_{j}}^{p_j,r_j} (x+\mu+\tau_{j}) 
\big\|_{ \ell^{2}_{\tau_{j}} ( \Lambda_{R_{j}} ) }
\Big\|
\varphi \Big( \frac{\cdot - (\nu_1 +\dots+ \nu_N) }{R_0} \Big) 
h \Big\|_{ L^2 (\R^n) } .
\end{align*}
Since \eqref{criticalvecm1} and \eqref{criticalvecm2} imply respectively that
$-n/2 < m_j - \alpha(p_j) < 0$
and $\sum_{j=1}^{N} (m_j - \alpha(p_j)) = n/2 - Nn/2$,
we have
by Lemma \ref{productLweakp'} with $r=2$
\begin{align*}&
|I| \lesssim
| \sigma |_{ \vecm }
\sum_{\mu \in \Z^n}
\langle \mu \rangle^{-L} 
\Big\|
\varphi \Big( \frac{\cdot - \nu }{R_0} \Big) 
h \Big\|_{ L^2 \ell^2_{\nu} } 
\prod_{j=1}^{N}
\big\| {F}_{\nu_{j}}^{p_j,r_j} (x+\mu+\tau_{j}) 
\big\|_{ \ell^{2}_{\tau_{j}} ( \Lambda_{R_{j}} ) \ell^2_{\nu_j} } ,
\end{align*}
and further have
by using that 
$\| \varphi ( ({x-\nu})/{R_0}) \|_{\ell^2_{\nu}} \lesssim R_{0}^{n/2}$
for any $x \in \R^n$
\begin{align} \label{Iketsuron}
|I| \lesssim
R_0^{n/2}
\| h \|_{L^2}
| \sigma |_{ \vecm }
\sum_{\mu \in \Z^n}
\langle \mu \rangle^{-L} 
\prod_{j=1}^{N}
\big\| {F}_{\nu_{j}}^{p_j,r_j} (x+\mu+\tau_{j}) 
\big\|_{ \ell^{2}_{\tau_{j}} ( \Lambda_{R_{j}} ) \ell^2_{\nu_j} } .
\end{align}

Collecting 
\eqref{hajime},
\eqref{Idef}, and
\eqref{Iketsuron},
we obtain
\begin{align*}
\| T_{ \sigma }
( f_1, \dots, f_N ) \|_{h^p}
\lesssim 
R_0^{n/2}
| \sigma |_{ \vecm }
\Big\| 
\sum_{\mu \in \Z^n}
\langle \mu \rangle^{-L}
\prod_{j=1}^{N}
\big\| {F}_{\nu_{j}}^{p_j,r_j} (x+\mu+\tau_{j}) 
\big\|_{ \ell^{2}_{\tau_{j}} ( \Lambda_{R_{j}} ) \ell^2_{\nu_j} }
\Big\|_{L^{p}_{x}}
.
\end{align*}
Apply the embedding
$\ell^{\min\{1,p\}} \hookrightarrow \ell^{1}$
to the sum over $\mu$
and choose $L > n/\min\{1,p\}$.
Then by Minkowski's inequality
the $L^{p}_{x}$ quasi-norm above
is bounded by
\begin{align*}
&
\bigg\| \langle \mu \rangle^{-L} \Big\| 
\prod_{j=1}^{N}
\big\| {F}_{\nu_{j}}^{p_j,r_j} (x+\mu+\tau_{j}) 
\big\|_{ \ell^{2}_{\tau_{j}} ( \Lambda_{R_{j}} ) \ell^2_{\nu_j} }
\Big\|_{L^{p}_{x} }
\bigg\|_{\ell^{\min\{1,p\}}_{\mu} }
\\&
\approx
\Big\| 
\prod_{j=1}^{N}
\big\| {F}_{\nu_{j}}^{p_j,r_j} (x+\tau_{j}) 
\big\|_{ \ell^{2}_{\tau_{j}} ( \Lambda_{R_{j}} ) \ell^2_{\nu_j} }
\Big\|_{L^{p}_{x}}
.
\end{align*}
We take 
$\widetilde{p}_{j} \in (0,\infty]$, $j=1,\dots,N$,
satisfying \eqref{pjs}
and use H\"older's inequality
to have
\begin{align} \label{0-2shiage}
\begin{split}
\| T_{ \sigma }
( f_1, \dots, f_N ) \|_{h^p}
\lesssim 
R_0^{n/2}
| \sigma |_{ \vecm }
\prod_{j=1}^{N}
\big\| {F}_{\nu_{j}}^{p_j,r_j} (x+\tau_{j}) 
\big\|_{ \ell^{2}_{\tau_{j}} ( \Lambda_{R_{j}} ) \ell^2_{\nu_j} L^{ \widetilde{p}_{j} }_{x} }
.
\end{split}
\end{align}
Using that
$\ell^{\min \{ 2, p_j \}}_{\tau_{j}} \hookrightarrow \ell^{2}_{\tau_{j}}$,
since $\min\{2,p_j\} \le \widetilde{p}_j$,
we have by Minkowski's inequality
\begin{align*}
\| {F}_{\nu_{j}}^{p_j,r_j} (x+\tau_{j}) 
\|_{ \ell^{2}_{\tau_{j}} ( \Lambda_{R_{j}} ) \ell^2_{\nu_j} L^{ \widetilde{p}_{j} }_{x} } 
&\le
\| {F}_{\nu_{j}}^{p_j,r_j} (x) 
\|_{ \ell^2_{\nu_j} L^{ \widetilde{p}_{j} }_{x}  \ell^{\min \{ 2, p_j \}}_{\tau_{j}} ( \Lambda_{R_{j}} )} 
\\&
\approx
R_{j}^{\max\{n/2, n/p_j\}}
\| {F}_{\nu_{j}}^{p_j,r_j} 
\|_{ \ell^2_{\nu_j} L^{ \widetilde{p}_{j} } } .
\end{align*}
We now recall that $r_j \in (0,\infty)$ can be chosen arbitrarily
and we take $r_j < \min\{ 1, \widetilde{p}_{j} \}$.
Then
\begin{align*}
\| {F}_{\nu_{j}}^{p_j,r_j} 
\|_{ \ell^2_{\nu_j} L^{ \widetilde{p}_{j} } }
&=
\| S_{r_j} ( \langle \nu_j \rangle^{ \alpha(p_j) } \square_{\nu_{j}} f_{j} )
\|_{ \ell^2_{\nu_j} L^{ \widetilde{p}_{j} } }
\\& 
\lesssim
\| \langle \nu_j \rangle^{ \alpha(p_j) } \square_{\nu_{j}} f_{j}(x)
\|_{ \ell^2_{\nu_j} L^{ \widetilde{p}_{j} }_{x} } 
\lesssim
\| f_{j} \|_{ W^{ \widetilde{p}_{j}, 2 }_{\alpha(p_j)} } .
\end{align*}
To the last quantity, we use \eqref{emb-WW} with $p_j \le \widetilde{p}_j$,
and then use \eqref{emb-hW}-\eqref{emb-bmoW}
to obtain
\[
\| f_{j} \|_{ W^{ \widetilde{p}_{j}, 2 }_{\alpha(p_j)} }
\lesssim
\| f_{j} \|_{ W^{ {p}_{j}, 2 }_{\alpha(p_j)} }
\lesssim
\| f_{j} \|_{ h^{{p}_{j}} } ,
\]
where $h^{p_j}$ can be replaced by $bmo$ when $p_j = \infty$.
Hence,
\begin{align} \label{fjSJ}
\| {F}_{\nu_{j}}^{p_j,r_j} (x+\tau_{j}) 
\|_{ \ell^{2}_{\tau_{j}} ( \Lambda_{R_{j}} ) \ell^2_{\nu_j} L^{ \widetilde{p}_{j} }_{x} } 
\lesssim
R_{j}^{\max\{n/2, n/p_j\}}
\| f_{j} \|_{ h^{{p}_{j}} } .
\end{align}

Lastly, 
substituting \eqref{fjSJ}
into \eqref{0-2shiage},
we obtain
\begin{align*}
\| T_{\sigma} ( f_1, \dots, f_N ) \|_{h^p}
\lesssim 
| \sigma |_{ \vecm }
R_0^{n/2}
\prod_{j=1}^{N}
R_{j}^{\max\{n/2, n/p_j\}}
\| f_{j} \|_{ h^{ {p}_{j} } } ,
\end{align*}
which completes the proof of Proposition \ref{main-prop} (1)-(a).

\subsection{Proof of Proposition \ref{main-prop} (1)-(b)}

We take a function $\varphi \in \calS(\R^n)$
satisfying $\varphi = 1$ on 
$\{ \zeta\in\R^n: |\zeta|\lesssim 1\}$.
By \eqref{decomposedT} and \eqref{suppFT},
we have 
\begin{align} \label{1-bsaisho}
\begin{split}
&
\| T_{ \sigma } ( f_1, \dots, f_N ) \|_{L^p}
=
\sup_{ h \in L^{p'} }
\bigg| \sum_{\vecnu\in(\Z^n)^N} \int_{\R^n}
T_{ \sigma_{\vecnu} }( \square_{\nu_1}f_1, \dots, \square_{\nu_N}f_N )(x) 
h(x) \, dx \bigg|
\\&=
\sup_{ h \in L^{p'} }
\bigg| \sum_{\vecnu\in(\Z^n)^N} \int_{\R^n}
T_{ \sigma_{\vecnu} }( \square_{\nu_1}f_1, \dots, \square_{\nu_N}f_N )(x) 
\;
\varphi \Big( \frac{D+ \nu_1 +\dots+ \nu_N }{R_0} \Big) 
h(x) \, dx \bigg| .
\end{split}
\end{align}
In what follows, we consider
\[
\II := 
\sum_{\vecnu\in(\Z^n)^N} \int_{\R^n}
T_{ \sigma_{\vecnu} }( \square_{\nu_1}f_1, \dots, \square_{\nu_N}f_N )(x) 
\;
\varphi \Big( \frac{D+ \nu_1 +\dots+ \nu_N }{R_0} \Big) 
h(x) \, dx .
\]
By Lemma \ref{main-lemma} with 
$W(\vecxi) = \prod_{j=1}^{N} \langle \xi_j \rangle^{m_j}$ and 
the notations 
\eqref{sigmavecm} and \eqref{tildeFnuj},
we have
\begin{align*}
&
|\II| \lesssim
|\sigma|_{ \vecm }
\sum_{\vecnu\in(\Z^n)^N} 
\int_{\R^n}
\prod_{j=1}^{N}
\langle \nu_j \rangle^{m_j}
\| S_{r_j} ( \square_{\nu_{j}} f_{j} ) (x+\tau_{j}) 
\|_{ \ell^{2}_{\tau_{j}} ( \Lambda_{R_{j}} ) } 
\;
\Big| \varphi \Big( \frac{D+ \nu_1 +\dots+ \nu_N }{R_0} \Big) h(x) \Big|
\, dx 
\\&=
|\sigma|_{ \vecm }
\int_{\R^n}
\sum_{\vecnu\in(\Z^n)^N} 
\prod_{j=1}^{N}
\langle \nu_j \rangle^{m_j-\alpha(p_j)}
\| {F}_{\nu_{j}}^{p_j,r_j} (x+\tau_{j}) 
\|_{ \ell^{2}_{\tau_{j}} ( \Lambda_{R_{j}} ) }
\Big| \varphi \Big( \frac{D+ \nu_1 +\dots+ \nu_N }{R_0} \Big) h(x) \Big|
\, dx .
\end{align*}
Note here that 
\eqref{criticalvecm1} and \eqref{criticalvecm2} 
imply respectively that
$-n/2 < m_j - \alpha(p_j) < 0$
and $\sum_{j=1}^{N} (m_j - \alpha(p_j)) = n/p - Nn/2$.
Then, using Lemma \ref{productLweakp'} with $r=p$
and H\"older's inequality with $\widetilde{p}_j \in (0,\infty]$ satisfying \eqref{pjs},
the integral above is estimated by
\begin{align*}&
\int_{\R^n}
\Big\| \varphi \Big( \frac{D+ \nu }{R_0} \Big) h(x) \Big\|_{ \ell^{p}_{\nu} }
\prod_{j=1}^{N}
\| {F}_{\nu_{j}}^{p_j,r_j} (x+\tau_{j}) 
\|_{ \ell^{2}_{\tau_{j}} ( \Lambda_{R_{j}} ) \ell^2_{\nu_j} }
\, dx 
\\&\le
\Big\| \varphi \Big( \frac{D+ \nu }{R_0} \Big) h \Big\|_{ \ell^{p}_{\nu} L^{p'} }
\prod_{j=1}^{N}
\| {F}_{\nu_{j}}^{p_j,r_j} (x+\tau_{j}) 
\|_{ \ell^{2}_{\tau_{j}} ( \Lambda_{R_{j}} ) \ell^2_{\nu_j} L^{ \widetilde{p}_{j} }_{x} }
 ,
\end{align*}
which implies, 
from Lemma \ref{MT-lemma24}, 
that
\begin{align*} 
|\II| \lesssim
R_0^{n/p} \|h\|_{ L^{p'} } 
|\sigma|_{ \vecm }
\prod_{j=1}^{N}
\| {F}_{\nu_{j}}^{p_j,r_j} (x+\tau_{j}) 
\|_{ \ell^{2}_{\tau_{j}} ( \Lambda_{R_{j}} ) \ell^2_{\nu_j} L^{ \widetilde{p}_{j} }_{x} } .
\end{align*}
We repeat the argument as in \eqref{fjSJ},
where we avoid the use of the embedding
$\ell^{\min \{ 2, p_j \}}_{\tau_{j}} \hookrightarrow \ell^{2}_{\tau_{j}}$
since $2 \le \widetilde{p}_j \leq \infty$ for $2 \le p < \infty$.
Then, 
\begin{align*} 
\| {F}_{\nu_{j}}^{p_j,r_j} (x+\tau_{j}) 
\|_{ \ell^{2}_{\tau_{j}} ( \Lambda_{R_{j}} ) \ell^2_{\nu_j} L^{ \widetilde{p}_{j} }_{x} } 
\lesssim
R_{j}^{n/2}
\| f_{j} \|_{ W^{ \widetilde{p}_{j}, 2 }_{ \alpha(p_j) } } 
\lesssim
R_{j}^{n/2}
\| f_{j} \|_{ W^{ {p}_{j}, 2 }_{ \alpha(p_j) } }
\lesssim
R_{j}^{n/2}
\| f_{j} \|_{ h^{ {p}_{j} } } .
\end{align*}
Gathering the above inequalities, we obtain
\begin{align*}
|\II| \lesssim 
| \sigma |_{ \vecm }
R_0^{n/p}
\prod_{j=1}^{N}
R_{j}^{n/2}
\| f_{j} \|_{ h^{ {p}_{j} } }
\|h\|_{ L^{p'} } ,
\end{align*}
where $h^{p_j}$ can be replaced by $bmo$ when $p_j = \infty$.
This completes the proof of Proposition \ref{main-prop} (1)-(b)
by combining with \eqref{1-bsaisho}.

\subsection{Proof of Proposition \ref{main-prop} (2)}

By using \eqref{decomposedT} and Lemma \ref{main-lemma}
with $W(\vecxi) = \langle \vecxi \rangle^{m}$,
\begin{align*}&
| T_{ \sigma } ( f_1, \dots, f_N ) (x) |
\le
\sum_{\vecnu\in(\Z^n)^N} 
\left|
T_{ \sigma_{\vecnu} }( \square_{\nu_1}f_1, \dots, \square_{\nu_N}f_N )(x) 
\right|
\\&\lesssim
|\sigma|_{m}
\sum_{\vecnu\in(\Z^n)^N} 
\big(1+|\nu_1|+\dots+|\nu_N| \big)^m
\prod_{j=1}^{N}
\left\|
S_{r_j} ( \square_{\nu_{j}} f_{j} ) (x+\tau_{j}) 
\right\|_{ \ell^{2}_{\tau_{j}} ( \Lambda_{R_{j}} ) } ,
\end{align*}
where $|\sigma|_{m}$ is as in \eqref{sigmam}.
Since $m=-Nn/2+\sum_{j=1}^{N} \alpha(p_j)$,
where $\alpha(p_j)$ is as in \eqref{tildeFnuj},
the sum over $\vecnu$ is bounded by
\begin{align*}
\sum_{\vecnu\in(\Z^n)^N} 
\big(1+|\nu_1|+\dots+|\nu_N| \big)^{-Nn/2}
\prod_{j=1}^{N}
\| {F}_{\nu_{j}}^{p_j,r_j} (x+\tau_{j}) 
\|_{ \ell^{2}_{\tau_{j}} ( \Lambda_{R_{j}} ) } .
\end{align*}
By Lemma \ref{productLweak1} and the argument as in \eqref{fjSJ},
the sum above is further estimated by
\begin{align*}
\prod_{j=1}^{N}
\| {F}_{\nu_{j}}^{p_j,r_j} (x+\tau_{j}) 
\|_{ \ell^{2}_{\tau_{j}} ( \Lambda_{R_{j}} ) \ell^2_{\nu_j} } 
&\le
\prod_{j=1}^{N}
\| {F}_{\nu_{j}}^{p_j,r_j} (x+\tau_{j}) 
\|_{ \ell^{2}_{\tau_{j}} ( \Lambda_{R_{j}} ) \ell^2_{\nu_j} L^{ \infty }_{x} } 
\lesssim
\prod_{j=1}^{N} R_j^{n/2}
\| f_{j} \|_{ W^{ \infty, 2 }_{ \alpha(p_j) } } .
\end{align*}
Here, since $p_{j} \le \infty$, $j =1,\dots,N$, 
the embeddings \eqref{emb-WW}-\eqref{emb-bmoW} 
of Lemma \ref{Waembd} yield that
\[
\| f_{j} \|_{ W^{ \infty, 2 }_{ \alpha(p_j) } } 
\lesssim
\| f_{j} \|_{ W^{ p_j, 2 }_{ \alpha(p_j) } } 
\lesssim
\| f_{j} \|_{ h^{p_j} } .
\]
Therefore,
collecting the estimates above, we obtain
\begin{align*}&
\| T_{ \sigma } ( f_1, \dots, f_N ) \|_{ L^{\infty} }
\lesssim
|\sigma|_{m}
\prod_{j=1}^{N}
R_{j}^{ n/2 }
\left\| f_{j} \right\|_{ h^{ {p}_{j} } }
\end{align*}
with $h^{p_j}$ replaced by $bmo$ when $p_j = \infty$,
which completes the proof of Proposition \ref{main-prop} (2).

\section{Sharpness}
\label{secSharp}

In this section, we consider the sharpness of the conditions
of the order $m \in \R$ and the smoothness $\vecs=(s_0,s_1, \dots, s_N) \in [0,\infty)^{N+1}$ 
stated in Theorem \ref{main-thm-1}.

\subsection{Sharpness of $m$ of Theorem \ref{main-thm-1}}
In this subsection, we show the following.
\begin{prop}
Let $N \ge 2$,  
$p, p_1,\dots, p_N \in (0, \infty]$,  
$1/p \le 1/p_1 + \cdots + 1/p_N$,  
$m\in \R$,
$s_0,s_1,\dots,s_N \in [0,\infty)$, and
$t \in (0,\infty]$.
If 
\begin{equation*}
\op ( S^{m}_{0,0} ( \vecs, t ; \R^n, N) ) 
\subset 
B(H^{p_1} \times \cdots \times H^{p_N} \to L^{p}), 
\end{equation*}
with $L^p$ replaced by $BMO$ when $p=\infty$, 
then
\begin{equation}\label{criticalm}
m \le 
\min \Big\{ \frac{n}{p}, \frac{n}{2} \Big\} 
- \sum_{j =1}^{N} 
\max \Big\{ \frac{n}{p_j}, \frac{n}{2} \Big\}.  
\end{equation}
\end{prop}

This is immediately obtained by
the inclusion $S^{m}_{0,0}(\R^n, N) \subset 
S^{m}_{0,0} ( \vecs, t ; \R^n, N)$
stated in Lemma \ref{classicalderivative} and
the following theorem proved in 
\cite[Theorem 1.5]{KMT-JFA}.

\begin{thm} 
Let $N \ge 2$,  
$0 < p, p_1,\dots, p_N \le \infty$,  
$1/p \le 1/p_1 + \cdots + 1/p_N$, and  
$m\in \R$. 
If 
\begin{equation*}
\op ( S^{m}_{0,0}(\R^n, N) ) 
\subset 
B(H^{p_1} \times \cdots \times H^{p_N} \to L^{p}), 
\end{equation*}
with $L^p$ replaced by $BMO$ when $p=\infty$, 
then \eqref{criticalm} holds.
\end{thm}

\subsection{Sharpness of $s_0$ of Theorem \ref{main-thm-1}}
In this subsection, we show the following.

\begin{prop}\label{sharpnessC}
Let $N \ge 2$,  
$p, p_1,\dots, p_N \in (0, \infty]$,
$s_0,s_1,\dots,s_N \in [0,\infty)$,
$t \in (0,\infty]$, and
\begin{equation} \label{mequ}
m=
\min \Big\{ \frac{n}{p}, \frac{n}{2} \Big\} 
- \sum_{j =1}^{N} 
\max \Big\{ \frac{n}{p_j}, \frac{n}{2} \Big\} .
\end{equation}
Suppose that the estimate 
\begin{equation} \label{sharpness-assump-s_0-t}
\begin{split}
\|T_{\sigma} 
\|_{ H^{p_1} \times \dots \times H^{p_N} \to L^p}
\lesssim 
\big\| 2^{\veck\cdot\vecs}
\| \langle \vecxi \rangle^{-m} \Delta_{\veck}\sigma(x,\vecxi)
\|_{L^{\infty}_{x,\vecxi} ((\R^n)^{N+1}) }
\big\|_{\ell^t_{\veck} ( (\N_0)^{N+1} ) }
\end{split}
\end{equation}
holds for all smooth functions $\sigma$ 
with the right hand side finite,
where $L^p$ is replaced by $BMO$ 
for $p=\infty$.
Then $s_0 \ge \min \{ n/p, n/2 \} $. 
\end{prop}

To show this, 
we will use the following lemma
which was given by Wainger \cite[Theorem 10]{Wainger} and 
by Miyachi and Tomita \cite[Lemma 6.1]{MT-IUMJ}.

\begin{lem} \label{sharpness-lem}
Let $0<a<1$, $0<b<n$,
and $\varphi \in \calS(\R^n)$.
For $\epsilon>0$, set 
\begin{equation*} 
f_{a, b, \epsilon}(x)
=
\sum_{k \in \Z^n \setminus \{0\}}
e^{-\epsilon|k|}|k|^{-b}e^{i|k|^a}e^{i k\cdot x}\varphi(x).
\end{equation*}
If $1\le p \le \infty$ and $b>n-an/2-n/p+an/p$, then
$\sup_{\epsilon>0}\|f_{a,b,\epsilon}\|_{L^p(\R^n)}<\infty$.
\end{lem}

Now, let us begin with the proof of Proposition \ref{sharpnessC}.
See also \cite[Proposition 7.3]{KMT-JMSJ}.

\begin{proof}
In this proof, for $p_j \in (0, \infty]$, 
we define the sets $J$ and $J^{c}$ by
\begin{align*} 
J=
\left\{ 
j \in \{1, \dots, N\} 
: 2 \le p_j \le \infty
\right\},
\quad
J^{c}=
\left\{ 
j \in \{1, \dots, N\} 
: 0<p_j <2
\right\} .
\end{align*}

It is sufficient to show that 
the condition $s_0 \ge \min \{ n/p, n/2 \} $ is deduced 
under the assumption \eqref{sharpness-assump-s_0-t} with $t=\infty$. 
In fact, once this is proved, 
then replacing  
$s_j$ by $s_j+\epsilon$, $\epsilon>0$,
$j=0,1,\dots,N$, 
we see that \eqref{sharpness-assump-s_0-t} with $t\in (0, \infty)$ 
implies 
$s_0+\epsilon \ge \min \{ n/p, n/2 \} $. 
Thus since $\epsilon >0$ is arbitrary, 
we must have $s_0 \ge \min \{ n/p, n/2 \} $.

Suppose 
\eqref{sharpness-assump-s_0-t} holds with $t=\infty$. 
For $\delta_1, \delta_2 > 0$,
we take functions $\varphi, \psi \in \calS(\R^n)$ such that
\begin{align}
&\label{conditionofphi}
\supp \varphi \subset \{ \xi \in \R^n: |\xi| \le \delta_1 \}, 
\quad
\int \varphi \neq 0,
\quad
\int \varphi^2 \neq 0,
\\&\nonumber
\supp \psi \subset \{ \xi \in \R^n: 2^{-1/2-\delta_2} \le |\xi| \le 2^{1/2+\delta_2} \}, 
\\&\nonumber
\psi = 1 \quad\text{on}\quad \{ \xi \in \R^n: 2^{-1/2+\delta_2} \le |\xi| \le 2^{1/2-\delta_2} \}
.
\end{align}
(We may note that $\calF^{-1}{\psi}$ has integral zero.)
For $\delta_3 > 0$ and $A \in \N$,
we set
\begin{align*}
D_{A} =
\{ \ell \in \Z^n:
2^{A-\delta_3} \le |\ell| \le 2^{A+\delta_3} \}. 
\end{align*}
Here, notice that
there exist $\delta_1, \delta_2, \delta_3 > 0$ such that
for any $A \in \N$
\begin{equation} \label{psiphi=phi}
\psi(2^{-A} \cdot) = 1
\quad\text{on}\quad
\supp \varphi(\cdot-\ell)
\;\;\text{with}\;\;
\ell \in D_{A} 
\end{equation}
(for instance, take 
$\delta_1 = 2^{-10}$, $\delta_2 = 2^{-2}$, 
and $\delta_3 = 2^{-3}$).
For $A \in \N$ and $\epsilon>0$
we set
\begin{align*}
&\sigma_{A} (x,\vecxi)
=\varphi(x)
e^{-ix \cdot (\xi_1+\dots+\xi_N)}
\sum_{ \ell_1,\dots, \ell_N \in D_{A} }
\langle \vecell \rangle^{m-s_0}
\Big( \prod_{j \in J} e^{-i|\ell_j|^{a_j}} \Big)
\Big( \prod_{j=1}^N \varphi(\xi_j-\ell_j) \Big),
\\
&f_{a_j,b_j, \epsilon }(x)
=\sum_{\ell_j \in \Z^n \setminus \{0\}} e^{-\epsilon|\ell_j|}
|\ell_j|^{-b_j}e^{i|\ell_j|^{a_j}}e^{i \ell_j \cdot x}
\calF^{-1}{\varphi}(x),
\quad j \in J,
\\
&f_{j,A}(x)= 2^{A n/p_j}
(\calF^{-1}{\psi})(2^{A} x),
\quad j \in J^{c},
\end{align*}
where 
$\vecell=(\ell_1,\dots,\ell_{N}) \in (\Z^n)^N$,
$0<a_j<1$, and
$b_j=n-a_j n/2-n/p_j+a_j n/p_j+\varepsilon_j$ 
with $\varepsilon_j>0$.
Here, we choose sufficiently small $\varepsilon_j>0$ satisfying $0<b_j<n$.

Firstly, we show that
\begin{align} 
& \label{sigma}
\| 2^{\veck\cdot\vecs}
\langle \vecxi \rangle^{-m} 
\Delta_{\veck} \sigma_{A} (x,\vecxi)
\|_{L^{\infty}_{x,\vecxi} \ell^{\infty}_{\veck} } \lesssim 1,
\\
&\label{fjt}
\|f_{a_j,b_j, \epsilon}\|_{H^{p_j}} \lesssim 1, 
\qquad  j \in J,
\\
&\label{fjell}
\|f_{j,A}\|_{H^{p_j}} \lesssim 1, 
\qquad  j \in J^{c} ,
\end{align}
with the implicit constants independent of $A \in \N$ and $\epsilon>0$.
Since $H^{p} = L^{p}$ for $2 \le p \le \infty$, 
\eqref{fjt} follows from Lemma \ref{sharpness-lem}.
Since $\|f_{j,A}\|_{H^{p_j}} = \| \calF^{-1}{\psi} \|_{H^{p_j}}$,
\eqref{fjell} holds.
In what follows, 
we shall consider \eqref{sigma}.
Let $L_j$ be a nonnegative integer
satisfying $L_j \ge s_j$, $j=0,1,\dots,N$.
Since 
\[
|\partial^{\alpha_0}_x
\partial^{\alpha_1}_{\xi_1} \dots \partial^{\alpha_N}_{\xi_N}
\sigma_{A} (x,\vecxi)|
\le
C_{\alpha_0,\alpha_1,\dots,\alpha_N}
\langle \vecxi \rangle^{m-s_0+|\alpha_0|} ,
\]
we see from the Taylor expansion that
\[
|\Delta_{\veck} \sigma_{A} (x,\vecxi)|
\lesssim
\begin{cases}
\langle \vecxi \rangle^{m-s_0}
2^{-k_1 L_1-\dots-k_N L_N},
\\
\langle \vecxi \rangle^{m-s_0+L_0}
2^{-k_0 L_0-k_1 L_1-\dots-k_N L_N}
\end{cases}
\]
(see \cite[Subsection 5.3]{KMT-JPDOA}).
By taking $0 \le \theta_0 \le 1$ satisfying $s_0=L_0\theta_0$,
we have
\begin{align} \label{Taylor-theta0}
\begin{split}&
|\Delta_{\veck} \sigma_{A} (x,\vecxi)|
=|\Delta_{\veck} \sigma_{A} (x,\vecxi)|^{1-\theta_0}
|\Delta_{\veck} \sigma_{A} (x,\vecxi)|^{\theta_0}
\\
&\lesssim
\left(\langle \vecxi \rangle^{m-s_0}
2^{-k_1L_1-\dots-k_NL_N}\right)^{1-\theta_0}
\left(\langle \vecxi \rangle^{m-s_0+L_0}
2^{-k_0L_0-k_1L_1-\dots-k_NL_N}\right)^{\theta_0}
\\
&=\langle \vecxi \rangle^{m}
2^{-k_0s_0-k_1L_1-\dots-k_NL_N} 
\end{split}
\end{align}
for any $\veck \in (\N_0)^{N+1}$.
Thus, we obtain \eqref{sigma}
with the implicit constant independent of $A \in \N$.

Choosing $\delta_1, \delta_2, \delta_3 > 0$ such that \eqref{psiphi=phi},
we have by the condition \eqref{conditionofphi}
\begin{align*}
T_{\sigma_{A}}(f_{\dots},\cdots,f_{\dots})(x)
&=
(2\pi )^{-Nn} 
\varphi(x)
\sum_{ \ell_1,\dots,\ell_N\in D_{A} }
\langle \vecell \rangle^{m-s_0}
\prod_{j \in J}
e^{ -\epsilon |\ell_{j}| }
|\ell_j|^{-b_j}
\int_{\R^n}\varphi(\xi_j-\ell_j)^2\, d\xi_j
\\& \qquad \qquad \times
\prod_{j \in J^{c}}
2^{A n(1/p_j-1)}
\int_{\R^n}
{\psi}(2^{- A} \xi_{j})
\varphi(\xi_j-\ell_j) \, d\xi_j
\\&=
C\varphi(x)
\sum_{ \ell_1,\dots,\ell_N \in D_{A} }
\langle \vecell \rangle^{m-s_0}
\prod_{j \in J} e^{ -\epsilon |\ell_{j}| } |\ell_j|^{-b_j}
\prod_{j \in J^{c}} 2^{A n(1/p_j-1)} .
\end{align*}
Hence,
collecting \eqref{sigma}, \eqref{fjt}, \eqref{fjell},
and the assumption \eqref{sharpness-assump-s_0-t} with $t=\infty$,
we see that 
\begin{equation*}
\sum_{ \ell_1,\dots,\ell_N \in D_{A} }
\langle \vecell \rangle^{m-s_0}
\prod_{j \in J} e^{-\epsilon|\ell_{j}|} |\ell_j|^{-b_j}
\prod_{j \in J^{c}} 2^{A n(1/p_j-1)} 
\lesssim 1
\end{equation*}
with the implicit constant independent of $\epsilon > 0$,
where we used that 
$\|\lambda f\|_{BMO} = |\lambda| \|f\|_{BMO}$,
$\lambda \in \R$,
when $p = \infty$.
Then, a limiting argument gives that
\begin{equation*}
\sum_{ \ell_1,\dots,\ell_N \in D_{A} }
\langle \vecell \rangle^{m-s_0}
\prod_{j \in J} |\ell_j|^{-b_j}
\prod_{j \in J^{c}} 2^{A n(1/p_j-1)} 
\lesssim 1,
\end{equation*}
and thus,
\begin{equation*} 
2^{ANn}
2^{A(m-s_0)}
\prod_{j \in J} 2^{-Ab_j}
\prod_{j \in J^{c}} 2^{A n(1/p_j-1)} 
\lesssim 1.
\end{equation*}
Since this holds for arbitrarily large $A \in \N$, we have
\[
Nn +(m-s_0) -\sum_{j \in J} b_j + \sum_{j \in J^{c}} \Big( \frac{n}{p_j} - n \Big) \le 0.
\]
Since $b_j \to n/2$ by taking the limits as $a_j \to 1$ and $\varepsilon_j \to 0$, 
we have
\[
s_0 \ge m +\sum_{j \in J} \frac{n}{2} + \sum_{j \in J^{c}} \frac{n}{p_j} ,
\]
which implies from \eqref{mequ} that $s_0 \ge \min \{ n/p, n/2 \} $.
This completes the proof.
\end{proof}

\subsection{Sharpness of $s_1,\dots,s_N$ of Theorem \ref{main-thm-1}}
We show that 
the conditions on $s_1,\dots,s_N$
stated in Theorems \ref{main-thm-1} are sharp.
See also 
\cite[Proposition 5.2]{KMT-JPDOA}
and \cite[Proposition 7.4]{KMT-JMSJ}.

\begin{lem}\label{sharp-smoothness-s_j-1}
Let $N \ge 2$,  
$p, p_1,\dots, p_N \in (0, \infty]$,  
$1/p = 1/p_1 + \cdots + 1/p_N$,  
$s_0,s_1,\dots,s_N \in [0,\infty)$,
$t \in (0,\infty]$, and $m \in \R$.
Suppose that the estimate 
\begin{equation}\label{sharpness-assump-s_j-t-n/pj}
\begin{split}
\|T_{\sigma} 
\|_{ H^{p_1} \times \dots \times H^{p_N} \to L^p}
\lesssim 
\big\|
2^{\veck\cdot\vecs}
\| \langle \vecxi \rangle^{-m} 
\Delta_{\veck}\sigma(x,\vecxi)
\|_{L^{\infty}_{x,\vecxi} ((\R^n)^{N+1}) }
\big\|_{\ell^t_{\veck} ((\N_0)^{N+1}) }
\end{split}
\end{equation}
holds for all smooth functions $\sigma$ 
with the right hand side finite,
where $L^p$ is replaced by $BMO$ 
for $p=\infty$.
Then $s_j \ge n/p_j$, $j = 1,\dots,N$. 
\end{lem}

\begin{proof}
We only prove $s_1 \ge n/p_1$
and the rest parts for $s_2,\dots,s_N$ follow by symmetry.

As stated in the proof of Proposition \ref{sharpnessC},
we may assume the assumption 
\eqref{sharpness-assump-s_j-t-n/pj} with $t=\infty$.
Take nonnegative functions $\varphi, \psi \in \calS(\R^n)$
satisfying $\int \varphi^2 \neq 0$ and
\begin{align} \begin{split} \label{psiphi} &
\supp \varphi \subset \{ \xi \in \R^n: |\xi| \le 2 \},
\quad
\varphi =1 \;\;\text{on}\;\; \{ \xi \in \R^n: |\xi| \le 1 \}, 
\\&
\supp \psi \subset \{ x \in \R^n: 1/2 \le |x| \le 2 \}.
\end{split} \end{align}
We set for $A \in \N$
\begin{align*}
&\sigma_{A} (x,\vecxi)
=2^{-As_1} \psi (2^{-A} x)
e^{-ix\cdot \xi_1}
\varphi(\xi_1)\dots\varphi(\xi_N),
\\&
{f_{1}}(x) =(\calF^{-1} \psi)(x),
\\&
{f_{j,A}}(x) =2^{-An/p_j} (\calF^{-1} \psi) (2^{-A} x),
\quad j = 2,\dots,N.
\end{align*}

Firstly we shall prove 
\begin{align}
&\label{sharpness-proof-s_j-1}
\| 2^{\veck\cdot\vecs}
\langle \vecxi \rangle^{-m}
\Delta_{\veck}\sigma_{A} (x,\vecxi)\|_{ L^{\infty}_{x,\vecxi} \ell^{\infty}_{\veck} }
\lesssim 1,
\\
&\label{sharpness-proof-s_j-2}
\| f_{1} \|_{H^{p_1}} 
\approx \| f_{j,A} \|_{H^{p_j}} 
\approx 1,
\quad
j=2,\dots,N,
\end{align}
for $A \in \N$.
By a scaling property of Hardy spaces
\eqref{sharpness-proof-s_j-2} obviously follows.
Let $L_j$ be a nonnegative integer
satisfying $L_j \ge s_j$ for $j=0,1,\dots,N$.
Observing that
\[
|\partial^{\alpha_0}_{x}\partial^{\alpha_1}_{\xi_1}
\dots
\partial^{\alpha_N}_{\xi_N}
\sigma_A(x,\vecxi)|
\le C_{\alpha_0, \alpha_1,\dots,\alpha_N}
\langle x \rangle^{-s_1+|\alpha_1|}
\langle \vecxi \rangle^{m},
\]
we see from the Taylor expansion that
\[
|\Delta_{\veck}\sigma_{A} (x,\vecxi)|
\lesssim
\begin{cases}
\langle x \rangle^{-s_1}
\langle \vecxi \rangle^{m} \; 
2^{-k_0L_0 -k_2L_2 - \dots - k_NL_N} ,
\\
\langle x \rangle^{-s_1+L_1}
\langle \vecxi \rangle^{m} \;
2^{-k_0L_0 -k_1L_1 -k_2L_2 - \dots -k_NL_N} ,
\end{cases}
\]
for any $\veck \in (\N_0)^{N+1}$.
As was done in \eqref{Taylor-theta0},
by taking $0\le \theta_1 \le 1$ such that $s_1=L_1\theta_1$,
\[
|\Delta_{\veck}\sigma_{A} (x,\vecxi)|
\lesssim
\langle \vecxi \rangle^{m}
2^{-k_0L_0 -k_1s_1 -k_2L_2 - \dots - k_NL_N},
\]
and thus, \eqref{sharpness-proof-s_j-1} follows
with the implicit constant independent of $A \in \N$.

From the condition of $\varphi$,
since 
$\psi(2^{A}\cdot)\varphi
=\psi(2^{A}\cdot)$ for $A\in\N$,
we have
\[
T_{\sigma_{A}}(f_1,f_{2,A},\dots,f_{N,A})(x)
=
C \,
2^{-As_1} \psi (2^{-A} x) \;
2^{-An(1/p_2 +\dots +1/p_N)}
\{ (\calF^{-1} \psi) (2^{-A} x) \}^{N-1}
\]
for $A \in \N$,
which implies, with the assumption $1/p=1/p_1+\dots+1/p_N$, that
\begin{align} \label{sharpness-proof-s_j-3}
\begin{split}
\| T_{\sigma_{A}}(f_1,f_{2,A},\dots,f_{N,A}) \|_{L^p}
\approx
2^{-As_1} 
2^{-An(1/p_2 +\dots +1/p_N)}
2^{An/p}
= 2^{-A(s_1-n/p_1)} ,
\end{split}
\end{align}
where we should use that $BMO$ is scaling invariant when $p=\infty$.

Thus, collecting 
\eqref{sharpness-proof-s_j-1}, 
\eqref{sharpness-proof-s_j-2},
\eqref{sharpness-proof-s_j-3},
and \eqref{sharpness-assump-s_j-t-n/pj} with $t=\infty$,
we see that
$2^{-A(s_1-n/p_1)} \lesssim 1$.
Since this holds for all $A \in \N$,
we obtain
$s_1 \ge n/p_1$,
which completes the proof.
\end{proof}

\begin{lem}\label{sharp-smoothness-s_j-2}
Let $N \ge 2$,  
$p, p_1,\dots, p_N \in (0, \infty]$,  
$1/p = 1/p_1 + \cdots + 1/p_N$,  
$s_0,s_1,\dots,s_N \in [0,\infty)$,
$t \in (0,\infty]$, and $m \in \R$.
Suppose that the estimate 
\begin{equation} \label{sharpness-assump-s_j-t-n/2}
\|T_{\sigma} 
\|_{ H^{p_1} \times \dots \times H^{p_N} \to L^p}
\lesssim 
\|\sigma\|_{ S^{m}_{0,0} (\vecs,t; \R^n, N) } 
\end{equation}
holds for all smooth functions $\sigma$ 
with the right hand side finite,
where $L^p$ is replaced by $BMO$ 
for $p=\infty$.
Then $s_j \ge n/2$, $j = 1,\dots,N$. 
\end{lem}

\begin{proof}
We only prove $s_1 \ge n/2$
and the rest parts for $s_2,\dots,s_N$ follow by symmetry.

We may assume the assumption 
\eqref{sharpness-assump-s_j-t-n/2} with $t=\infty$.
Let $\{ {\psi}_k^{\ast} \}_{k \in \N_0}$ be 
a radial Littlewood-Paley partition of unity on $\R^n$.
We take radial functions $\varphi, \psi \in \calS(\R^n)$
satisfying \eqref{psiphi}
and set for $A \in \N$
\begin{align*}&
\sigma_{A}(x,\vecxi)
=\sigma_{A}(\vecxi)
=\calF^{-1}[\psi (2^{-A} \cdot )](\xi_1)
\varphi(\xi_2) \dots \varphi(\xi_N), 
\\&
f_{j,A}(x) =
(\calF^{-1}\psi) (2^{-A}x),
\quad
j = 1,\dots,N.
\end{align*}

For these functions, the following hold:
\begin{align}
&\label{sharpness-proof-s_j-n/2-1}
\| \sigma_{A} \|_{ S^{m}_{0,0} (\vecs,\infty; \R^n, N) } 
\lesssim 2^{A(s_1+n/2)} ,
\\
&\label{sharpness-proof-s_j-n/2-2}
\| f_{j,A} \|_{H^{p_j}} 
\approx 2^{An/p_j},
\quad
j=1,\dots,N,
\end{align}
for $A \in \N$.
As in the previous proof
\eqref{sharpness-proof-s_j-n/2-2} is obvious.
Using the fact that
$\psi_{k_0}^{\ast}(D_x)[1]$ is equal to $1$ if $k_0=0$
and to $0$ if $k_0 \ge 1$, 
we have
\[
\Delta_{\veck}\sigma_{A} (x,\vecxi)
=
\calF^{-1}[ \psi_{k_1}^{\ast} \psi (2^{-A} \cdot )] (\xi_1)
\psi_{k_2}^{\ast}(D) \varphi (\xi_2)
\dots \psi_{k_N}^{\ast}(D) \varphi (\xi_N) .
\]
Here,
$\calF^{-1}[ \psi_{k_1}^{\ast} \psi (2^{-A} \cdot )]$ 
vanishes unless $|k_1 - A| \le 1$.
We have for sufficiently large $L_1>0$
\begin{align*}
| \calF^{-1}[ \psi_{k_1}^{\ast} \psi (2^{-A} \cdot )] (\xi_1) |
\lesssim 
2^{An} \langle 2^A \xi_1 \rangle^{-L_1} .
\end{align*}
Since $\varphi \in \calS(\R^n)$,
we have for sufficient large $M_{j} > 0 $ and $L_{j}>0$
\[
| \psi_{k_j}^{\ast}(D) \varphi (\xi_j) |
\lesssim
2^{-k_j M_j} \; \langle \xi_j \rangle^{-L_{j}},
\quad
j = 2,\dots,N .
\]
Hence, by the embedding $L^{2} \hookrightarrow L^{2}_{ul}$ and the inequality
$\langle \vecxi \rangle^{m} \lesssim \prod_{j=1}^{N} \langle \xi_j \rangle^{|m|}$, 
we see that
\begin{align*} &
\| \sigma_{A} \|_{ S^{m}_{0,0} (\vecs,\infty; \R^n, N) } 
\\&\lesssim
\sup_{ |k_1 - A| \le 1 } 
2^{ s_1 k_1} 
\| \langle \xi_1 \rangle^{|m|}
2^{An} \langle 2^{A} \xi_1 \rangle^{-L_1}
\|_{ L^2_{\xi_1} }
\prod_{j=2}^N
\sup_{k_j \in \N_0}
2^{ s_j k_j }
\| \langle \xi_j \rangle^{|m|}
2^{ -M_j k_j } \langle \xi_j \rangle^{-L_j}
\|_{ L^2_{\xi_j} }
\\&\lesssim
2^{A(s_1+n/2)} ,
\end{align*}
which gives \eqref{sharpness-proof-s_j-n/2-1}.
Moreover, since the conditions of $\varphi,\psi$ imply
$\psi(2^{A} \cdot) \varphi
=\psi(2^{A} \cdot)$, $A \in \N$,
\begin{align*}&
T_{\sigma_{A}}(f_{1,A},\dots,f_{N,A})(x)
=
2^{An} (\psi \ast \widehat\psi)(2^{-A}x)
\{ \widehat\psi(2^{-A}x) \}^{N-1} ,
\end{align*}
which implies that 
\begin{align} \label{sharpness-proof-s_j-n/2-3}
\| T_{\sigma_{A}}(f_{1,A},\dots,f_{N,A}) \|_{L^p}
\approx 
2^{An} \, 2^{An/p} ,
\end{align}
where we should use that $BMO$ is scaling invariant when $p=\infty$.

Thus, collecting 
\eqref{sharpness-proof-s_j-n/2-1}, 
\eqref{sharpness-proof-s_j-n/2-2},
\eqref{sharpness-proof-s_j-n/2-3},
and \eqref{sharpness-assump-s_j-t-n/2} with $t=\infty$
and using the assumption $1/p=1/p_1+\dots+1/p_N$,
we obtain
$s_1 \ge n/2$.
This completes the proof.
\end{proof}

The following immediately follows from 
Lemmas \ref{sharp-smoothness-s_j-1} and \ref{sharp-smoothness-s_j-2}.

\begin{cor} \label{sharp-smoothness-s_j}
Let $N \ge 2$,  
$p, p_1,\dots, p_N \in (0, \infty]$,  
$1/p = 1/p_1 + \cdots + 1/p_N$,  
$s_0,s_1,\dots,s_N \in [0,\infty)$,
$t \in (0,\infty]$, and $m \in \R$.
Suppose that the estimate 
\begin{equation*}
\|T_{\sigma} 
\|_{ H^{p_1} \times \dots \times H^{p_N} \to L^p}
\lesssim 
\|\sigma\|_{ S^{m}_{0,0} (\vecs,t; \R^n, N) } 
\end{equation*}
holds for all smooth functions $\sigma$ 
with the right hand side finite,
where $L^p$ is replaced by $BMO$ 
for $p=\infty$.
Then $s_j \ge \max\{n/p_j, n/2\}$, $j = 1,\dots,N$. 
\end{cor}


\section*{Acknowledgments}
The author sincerely expresses deep thanks 
to Prof. A. Miyachi and Prof. N. Tomita.
Although Proposition \ref{sharpnessC} in the first draft 
stated only the case $p_{j} \ge 1$,
Prof. Miyachi gave him ideas to develop it to the whole range $p_{j} > 0$.
Prof. Tomita pointed out to him that 
Theorem \ref{main-thm-2} holds for more improved symbol classes
as stated in Remark \ref{Remark-class}.



\end{document}